\documentclass[a4paper]{cas-sc}
\usepackage[authoryear,longnamesfirst]{natbib}

\usepackage[utf8]{inputenc} 
\usepackage[T1]{fontenc}    

\usepackage{hyperref}
\usepackage{booktabs}       
\usepackage{nicefrac}       
\usepackage{microtype}      
\usepackage{xcolor}         

\usepackage{amsfonts,amsthm,array}

\usepackage{algorithm}
\usepackage[font=small]{caption}
\usepackage{subcaption}
\usepackage{color}
\usepackage{graphicx}
\usepackage{enumitem}
\usepackage{bbm}
\usepackage{threeparttable}
\usepackage{algorithmicx}
\usepackage{algpseudocode}

\usepackage{natbib}

\usepackage{mathtools}
\usepackage{latexsym}
\usepackage{epsfig}

\theoremstyle{plain}
\newtheorem{lemmasec}{Lemma}[section]
\newtheorem{theoremsec}{Theorem}[section]
\newtheorem{corollarysec}{Corollary}[section]

\newtheorem{theorem}{Theorem}
\newtheorem{corollary}{Corollary}
\newtheorem{assumption}{Assumption}
\newtheorem{lemma}{Lemma}

\newtheorem{definition}{Definition}
\newtheorem{example}{Example}

\renewcommand{\gg}{\textbf{g}}
\newcommand{\Exp}{\mathbb{E}}
\newcommand{\R}{\mathbb{R}}
\newcommand{\N}{\mathbb{N}}
\newcommand{\dotprod}[2]{\left\langle #1,#2 \right\rangle}
\newcommand{\norms}[1]{\left\| #1 \right\|}
\newcommand{\expect}[1]{\mathbb{E}\left[ #1 \right]}
\newcommand{\algname}[1]{{\sf  #1}\xspace}
\newcommand{\Prob}[1]{\mathbb{P} \left[ #1\right]}
\usepackage{dsfont}
\usepackage{pifont}

\def\obf{\mathds{1}}

\newcommand{\circledOne}{\text{\ding{172}}}
\newcommand{\circledTwo}{\text{\ding{173}}}
\newcommand{\circledThree}{\text{\ding{174}}}
\newcommand{\circledFour}{\text{\ding{175}}}
\newcommand{\circledFive}{\text{\ding{176}}}
\newcommand{\circledSix}{\text{\ding{177}}}

\usepackage{xspace}

\def  \R {\mathbb R}

\usepackage{amsmath}
\usepackage{amssymb}
\usepackage{bm}

\def\tsc#1{\csdef{#1}{\textsc{\lowercase{#1}}\xspace}}
\tsc{WGM}
\tsc{QE}
\tsc{EP}
\tsc{PMS}
\tsc{BEC}
\tsc{DE}

\begin{document}

\let\WriteBookmarks\relax
\def\floatpagepagefraction{1}
\def\textpagefraction{.001}

\shorttitle{Gradient-free algorithm for saddle point problems under overparametrization}

\shortauthors{E.~Statkevich et al.}

\title [mode = title]{Gradient-free algorithm for saddle point problems under overparametrization}                      



%
\author[1]{Ekaterina Statkevich }[orcid=0009-0002-7775-5762]
\fnmark[1]
\ead{statkevichk@bk.ru}

\author[1]{Sofiya Bondar }[orcid=0009-0009-1853-5656]
\fnmark[1]
\ead{sonbondar@yandex.ru}

\author[2]{Darina Dvinskikh }[orcid=0000-0003-1757-1021]
\ead{dmdvinskikh@hse.ru}

\author[1, 3, 4]{Alexander Gasnikov }[orcid=0000-0002-7386-039X]
\ead{gasnikov@yandex.ru}

\author[1, 3, 4]{Aleksandr Lobanov }[orcid=0000-0003-1620-9581]
\ead{lobbsasha@mail.ru}

\affiliation[1]{organization={Moscow Institute of Physics and Technology, Department of Applied Mathematics and Computer Science},
    addressline={Institutsky lane 9}, 
    city={Dolgoprudny},
    postcode={141701}, 
    country={Russian Federation}}

\affiliation[2]{organization={HSE University},
    addressline={Myasnitskaya st. 20}, 
    city={Moscow},
    postcode={101000}, 
    country={Russian Federation}}

\affiliation[3]{organization={Skolkovo Institute of Science and Technology
},
    addressline={Bolshoy Boulevard, 30, p.1}, 
    city={Moscow},
    postcode={121205}, 
    country={Russian Federation}}

\affiliation[4]{organization={ISP RAS Research Center for Trusted Artificial Intelligence},
    addressline={Alexander Solzhenitsyn st. 25}, 
    city={Moscow},
    postcode={109004}, 
    country={Russian Federation}}

\fntext[fn1]{Equal contribution.}

\begin{abstract}
This paper focuses on solving a stochastic saddle point problem (SPP) under an overparameterized regime for the case, when the gradient computation is impractical. As an intermediate step, we generalize Same-sample Stochastic Extra-gradient algorithm \cite{gorbunov2022stochastic} to a biased oracle and estimate novel convergence rates.  As the result of the paper we introduce an algorithm, which uses gradient approximation instead of a gradient oracle.
We also conduct an analysis to find the maximum admissible level of adversarial noise and the optimal number of iterations at which our algorithm can guarantee achieving the desired accuracy.
\end{abstract}




\maketitle

\section{Introduction} \label{sec:Introduction}
The situation when the gradient of the objective function is not available for any reason is actively studied in the literature of optimization for machine learning. This class of problems is often referred to as a black-box optimization problem \cite{kimiaei2022efficient,conn2009introduction}, where the black-box itself is a zero-order oracle \cite{rosenbrock1960automatic}. There are a number of works that focus on the applied setting of the problem, proposing gradient-free optimization algorithms as a solution to a narrower class of black-box optimization problem. For example, the works \cite{Patel_2022,Alashqar_2023} propose gradient-free convex optimization algorithms in the federated learning setting. In particular, \cite{Alashqar_2023} use the concept of an inexact oracle, i.e., a zero-order oracle outputs the function value at the requested point with bounded noise $\Tilde{f} = f(z,\xi) + \delta(z)$. Other works \cite{Gasnikov2023highly,sadykov2023gradient} focus on min and min-max problems, respectively, assuming that the objective function satisfies the Polyak-Lojasiewicz condition. The next few works \cite{akhavan2021distributed,lobanov2023non} consider the setting of distributed learning. And also the authors of \cite{lobanov2023accelerated} proposed a gradient-free algorithm for convex stochastic optimization problem under overparameterization condition.

Typically, perhaps the most common way to create such gradient-free algorithms among theoretical works is to take advantage of first-order algorithms \cite{gasnikov2022randomized}. The basic idea of such a technique is to use a gradient approximation instead of the true gradient of the function in a first-order algorithm. For example, the following first-order algorithms have been used to create the above gradient-free algorithms \cite{woodworth2021min,ajalloeian2020convergence,yang2020global,sayin2017stochastic,kovalev2021lower,woodworth2021even}. Since gradient approximations have not only bounded variance but also bias, we need to base the creation of a gradient-free algorithm on a biased first-order algorithm, i.e., with a biased gradient oracle.

When creating gradient-free algorithms, authors usually consider three optimality criteria: the number of sequential iterations $N$, the number of calls to the zero-order oracle $T$, and finally, the maximum noise level $\Delta$ at which the desired accuracy $\varepsilon$ can still be achieved. It is not hard to see that the first two criteria are also standard for first-order optimization algorithms, while the third criterion is some specificity for zero-order algorithms. There are many motivations for finding the maximum noise level, e.g., $|\delta (z)|\leq \Delta$, from resource conservation (as the noise level in the oracle increases, the cost of accessing it decreases) to cofidentiality (inability to communicate exact data due to company secrecy). As the work \cite{Gasnikov_2022_ICML} shows to achieve optimal estimates on iteration and oracle complexities, it is sufficient to take as a base an accelerated batched first-order algorithm in the desired setting, but the estimate on the maximum noise level depends on the analysis of variance and bias of the chosen gradient approximation. But already work \cite{lobanov2023black} has shown that by using the effect of overbatching, it is possible to improve the maximum noise level estimate (or for a fixed noise level, to improve the error floor or the asymptote to which the algorithm converges) while maintaining the optimal iterative complexity, but sacrificing the optimal oracle complexity.

In our work, we focus on solving the saddle point optimization problem \cite{beznosikov2023smooth}, assuming that the overparameterization condition is satisfied. We choose the Stochastic Extragradient algorithm from \cite{gorbunov2022stochastic} as the first-order algorithm on which we base our new gradient-free algorithm. However, the first-order algorithm in \cite{gorbunov2022stochastic} uses the exact gradient as the gradient oracle, then in this paper we generalize the algorithm to the case with a biased gradient oracle by providing new convergence estimates for different sample settings: Uniform Sampling setup and Importance Sampling setup. Using oracles with different types of adversarial noise (deterministic and stochastic), we provide convergence estimates for a new gradient-free algorithm (\algname{ZOS-SEG}, see Algorithm~\ref{alg:zosseg}). As a competitor, we consider the following works \cite{sadiev2021zeroth,dvinskikh2022noisy,sadykov2023gradient}. We perform comparisons of our algorithm with others on a numerical experiment.

\subsection{Contributions}
\begin{itemize}
    \item \textbf{Generalized analysis of S-SEG.} We generilized S-SEG algorithm to a biased oracle, which is applicable to stochastic optimization problems under overparametrization, and derived convergence rates (Theorem \ref{thm:S-SEG_convergence_rate}). We further specify this result via using Uniform Sampling based algorithm (Corollary \ref{cor:S-SEGUS-o-convergence}) and Importance Sampling based algorithm (Corollary \ref{cor:S-SEGIS-o-convergence}). The theoretical superiority of SEG-IS over SEG-US persists when biased oracle is added.
    
    \item \textbf{New gradient-free method.} By substituting gradient oracle for gradient approximation with a gradient-free oracle, we propose Zero-Order Same-sample Stochastic Extragradient algorithm (Algorithm \ref{alg:zosseg}),a gradient-free algorithm for saddle point problems (SPP) under overparametrized conditions. There are two variants of the realisation the algorithm - different sampling processes: Uniform Sampling setup and Importance Sampling setup. 
    We provide the theoretical base for both mentioned samplings, theoretical results are in Table \ref{tab:all_types}. The  distinction between the US setup and IS setup convergence rates is the fact that the estimation for IS setup will be lower due to the exponential term and terms that are quadratically dependent on $L_{\max}$, as in most cases $\overline{L}$ is less than $L_{\max}$. 
    
    \item \textbf{Deterministic and stochastic noises setup.} We estimated number of iterations of the algorithm, the smoothing parameter, the batch-size, adversarial noise and the total number of calls to the gradient-free oracle, required to reach $\varepsilon-$accuracy for both deterministic (Theorem \ref{th:zosseg-epsilon-acc}) and stochastic (Theorem \ref{th:zosseg-epsilon-acc-stochnoise}) oracle biases. Results might be seen in Table \ref{tab:convergencies}. By using the   stochastic noise in the definition of a gradient-free oracle, we have achieved a smaller number of terms in the convergence rate estimate of our algorithm. However, if we aim to reduce the noise, the case of deterministic noise is more advantageous.
    
    \item \textbf{Numerical evaluation.} In Section~\ref{sec:Experiments}, we corroborate our theoretical results with experimental testing. We use Uniform Sampling setup for both cases of noise. We compared our algorithm with several other algorithms, used for the solution of SPP and achieved much better convergence rate than mentioned algorithms.

\end{itemize}

\begin{table}[width=.9\linewidth,cols=4,pos=ht] 
\caption{Summary of the convergence rates for Algorithm~\ref{alg:zosseg}. Columns <<Deterministic>> and <<Stochastic>> indicate types of bounded noise added to the objective function $f(z)$. For Uniform Sampling we imply that $ \overline{\mu} = \frac{1}{n}\sum\limits_{i: \mu_i \geq 0} \mu_i + \frac{4}{n}\sum\limits_{i: \mu_i < 0} \mu_i$ and for Importance Sampling $\overline{L} = \frac{1}{n}\sum_{i=1}^n L_i$.}
\label{tab:all_types}
\begin{tabular*}{\tblwidth}{@{} LLLL@{} }
\toprule
\multirow{2}{*}{\makecell{Algorithm\\ distribution}} & \multicolumn{2}{c}{Gradient-free oracle noise} \\\\
 & Deterministic & Stochastic\\
\midrule
Uniform Sampling & \makecell[l]{$\frac{L_{\max}R_0^2}{\overline{\mu}}\exp\left(-\frac{\overline{\mu}N}{L_{\max}} \right)$ \\ $+ \frac{d \sigma_{*}^2}{\overline{\mu}^2 N B} + \frac{d L_{\max}^2 \tau^2}{\overline{\mu}^2 N B}$ \\ $+ \frac{d^2 \Delta^2}{\overline{\mu}^2 N B \tau^2} + \frac{L_{\max}^2 \tau^2}{\overline{\mu}^2} + \frac{d^2 \Delta^2}{\overline{\mu}^2 \tau^2}$} & \makecell[l]{$\frac{L_{\max}R_0^2}{\overline{\mu}}\exp\left(-\frac{\overline{\mu}N}{L_{\max}} \right)$ \\ $+ \frac{d \sigma_{*}^2}{\overline{\mu}^2 N B} + \frac{d L_{\max}^2 \tau^2}{\overline{\mu}^2 N B}$ \\ $+ \frac{d^2 \Delta^2}{\overline{\mu}^2 N B \tau^2}$} \\
\midrule
Importance Sampling & \makecell[l]{$\frac{\overline{L} R_0^2}{\overline{\mu}}\exp\left(-\frac{\overline{\mu}N}{\overline{L}} \right)$ \\ $+ \frac{\overline{L} d \sigma_{*}^2}{\overline{\mu}^2 N B L_{min}} + \frac{d \overline{L}^3 \tau^2}{\overline{\mu}^2 N B L_{min}}$ \\ $+ \frac{\overline{L} d^2 \Delta^2}{\overline{\mu}^2 N B \tau^2 L_{min}} + \frac{\overline{L}^2 \tau^2}{\overline{\mu}^2} + \frac{d^2 \Delta^2}{\overline{\mu}^2 \tau^2}$} & \makecell[l]{$\frac{\overline{L} R_0^2}{\overline{\mu}}\exp\left(-\frac{\overline{\mu}N}{\overline{L}} \right)$ \\ $+ \frac{\overline{L} d \sigma_{*}^2}{\overline{\mu}^2 N B L_{min}} + \frac{d \overline{L}^3 \tau^2}{\overline{\mu}^2 N B L_{min}}$ \\ $+ \frac{\overline{L} d^2 \Delta^2}{\overline{\mu}^2 N B \tau^2 L_{min}}$} \\
\bottomrule
\end{tabular*}
\end{table}

\begin{table}[width=.9\linewidth,cols=4,pos=ht]
\caption{Upper-bounds on all parameters required for convergence rate of Algorithm~\ref{alg:zosseg} on Uniform Sampling to achieve $\varepsilon$-accuracy. $N$ is number of iterations of the algorithm, $\tau$ is the smoothing parameter, $B$ is the batch-size, $\Delta$ is adversarial noise and $T$ is the total number of calls to the gradient-free oracle.}
\label{tab:convergencies}
\begin{tabular*}{\tblwidth}{@{} LLL@{} }
\toprule
\makecell{Parameters} & \multicolumn{2}{c}{Uniform sampling setup} \\\\
 & \makecell{Deterministic} & \makecell{Stochastic}\\
\midrule
N & $\mathcal{O}\left( \frac{L_{\max}}{\overline{\mu}} \ln \left( \frac{R_0^2 L_{\max}}{\varepsilon \overline{\mu}} \right) \right)$ & $\mathcal{O}\left( \frac{L_{\max}}{\overline{\mu}} \ln \left( \frac{R_0^2 L_{\max}}{\varepsilon \overline{\mu}} \right) \right)$\\ 
\midrule
$\tau$ & $  \frac{\overline{\mu}}{L_{\max}} \sqrt{\varepsilon}$ &  $ \max \left\{ \frac{\sigma_{*}}{L_{\max}}, \sqrt{\frac{\varepsilon \overline{\mu}}{L_{\max} d}  \ln \left( \frac{R_0^2 L_{\max}}{\varepsilon \overline{\mu}} \right)}\right\}$ \\
\midrule
$B$   & \scriptsize{$ \mathcal{O}\left( \max\left\{ \frac{d \overline{\mu}}{L_{\max} \ln \left( \frac{R_0^2 L_{\max}}{\varepsilon \overline{\mu}} \right)} \max \left\{ \frac{\sigma_{*}^2}{\varepsilon \overline{\mu}^2 } ,  1 \right\}, 1 \right\} \right) $} & \scriptsize{ $\mathcal{O}\left( \max\left\{ \frac{d \sigma_{*}^2}{\varepsilon \overline{\mu} L_{\max} \ln \left( \frac{R_0^2 L_{\max}}{\varepsilon \overline{\mu}} \right)} , 1 \right\} \right)$} \\
\midrule
$\Delta$ & \tiny{$  \frac{\varepsilon \overline{\mu}^2}{L_{\max} \sqrt{d}} \min \left\{ \max \left\{ 1 , \frac{ \sigma_{*}}{\overline{\mu} \sqrt{\varepsilon}}, \sqrt{\frac{L_{\max} \ln \left( \frac{R_0^2 L_{\max}}{\varepsilon \overline{\mu}} \right)}{d \overline{\mu}}} \right\}, \frac{1}{\sqrt{d}} \right\}$} & \tiny{$ \frac{\sqrt{\varepsilon}  \sigma_{*}}{L_{\max} d} \max \left\{ 1 , \frac{\overline{\mu}^2 \sqrt{\varepsilon} \ln \left( \frac{R_0^2 L_{\max}}{\varepsilon \overline{\mu}} \right)}{ \sigma_{*}  \sqrt{d}}, \sqrt{\frac{ \overline{\mu}^3  \ln \left( \frac{R_0^2 L_{\max}}{\varepsilon \overline{\mu}} \right)}{L_{\max}}},  \sqrt{\frac{ \varepsilon \overline{\mu} L_{\max} \ln \left( \frac{R_0^2 L_{\max}}{\varepsilon \overline{\mu}} \right)}{d}} \right\}$} \\
\midrule
$T$ & $ \mathcal{O}\left(\max \left\{ d, \frac{d \sigma_{*}^2}{\varepsilon \overline{\mu}^2}, \frac{L_{\max}}{\overline{\mu}} \ln \left( \frac{R_0^2 L_{\max}}{\varepsilon \overline{\mu}} \right) \right\} \right)$  & $ \mathcal{O}\left(\max \left\{ \frac{d \sigma_{*}^2}{\varepsilon \overline{\mu}^2}, \frac{L_{\max}}{\overline{\mu}} \ln \left( \frac{R_0^2 L_{\max}}{\varepsilon \overline{\mu}} \right) \right\} \right)$ \\
\bottomrule
\end{tabular*}
\end{table}

\newpage
\subsection{Paper Organization}
This paper has the following structure. In Section \ref{sec:biased_gradient}, we provide a SEG-US and SEG-IS algorithms with a biased gradient oracle in the overparameterization setup. In Section \ref{sec:Main_Result}, we present the main result of this paper for saddle point problems: our algorithm \algname{ZOS-SEG} and deduced estimations on convergence rate for deterministic and stochastic oracle noises. Section \ref{sec:Experiments} presents results of our algorithm in comparison with other methods. While Section \ref{sec:Conclusion} concludes this paper. Detailed proofs are presented in the supplementary materials (Appendix).

\subsection{Notations}
We use $\dotprod{z_1}{z_2}:= \sum_{i=1}^{d} z_{1i} z_{2i}$ to denote standard inner product of $z_1, z_2 \in \mathbb{R}^d$, where $z_{1i}$ and $z_{2i}$ are the $i$-th component of $z_1$ and $z_2$ respectively. We denote Euclidean norm ($l_2$-norm) in~$\mathbb{R}^d$ as $\norms{z} = \| z\|_2 := \sqrt{\dotprod{z}{z}}$.
We use the following notation $B^d(r):=\left\{ z \in \mathbb{R}^d : \| z \| \leq r \right\}$ to denote Euclidean ball ($l_2$-ball) and  $S^d(r):=\left\{ z \in \mathbb{R}^d : \| z \| = r \right\}$ to denote Euclidean sphere. Operator $\mathbb{E}[\cdot]$ denotes full mathematical expectation. We notate $\mathcal{O} (\cdot)$ to hide numerical constants. We use $\obf_{\{condition\}}$ to define indicator function from condition in brackets. 
We use $U[a; b]$ to denote continuous uniform distribution on interval $[a; b]$. We use $\{1..n\}$ to denote all natural numbers from 1 to $n$.

\section{Stochastic gradient-based algorithm with bias} \label{sec:biased_gradient}
In this section, we generilize SEG algorithm~\cite{gorbunov2022stochastic} to a biased gradient oracle. We also refer to our problem in the following form of a variational inequality to be more implicit about our results:

\begin{equation}\label{eq:variational-ineq-form}
    \text{Find } z^* \in \R^d \text{, such that } F(z^*) = 0,
\end{equation}
where operator $F: \mathbb{R}^d \rightarrow \mathbb{R}^d $ is a monotone operator, expressed as a finite sum, $F(z)~=~\frac{1}{n}\sum_{i=1}^n F_i(z)$ or more generally as expectation $F(z) = \Exp_{\xi \sim \mathcal{D}} \left[F_{\xi}(z)\right]$. In the case of gradient oracle for SPP operator  $F$ is defined as: $F_{\xi}(z) = \left[ \nabla_x f(x, y, \xi); -\nabla_y f(x, y, \xi)\right]$, where $z = (x, y)$.\\
Firstly, we need to introduce several definitions and assumptions.

\subsection{Preliminaries}
\begin{definition}[Biased oracle]\label{def:b_oracle}
In general, biased oracle $\gg_{\xi}(z)$ with \textit{bounded} noise in a stochastic problem may be viewed as
\begin{equation*}
    \gg_{\xi}(z) := \gg(z,\xi) = F_{\xi}(z) + b(z),
\end{equation*}
where $F_{\xi}: \R^d \rightarrow \mathbb{R}^d$ is a stochastic operator, $\xi$ is a random variable and $b: \R^d \rightarrow \R^d$ such that there exists $\zeta > 0:$ for~all $z \in \R^d:$ $\|b(z)\|^2 \leq \zeta^2$.
\end{definition}

As we are considering a general case, we need to introduce some assumptions on operator $F_{\xi}$. When the operator $F_{\xi}$ is monotone, it is known that the standard gradient method
does not converge without strong monotonicity~\citep{noor2003new} or cocoercivity~\citep{chen1997convergence}. Referring to \citep{gorbunov2022stochastic} we have two main assumptions on operator~$F_{\xi}$.
\begin{assumption}[Lipschitzness]\label{as:lipschitzness}
For~all $\xi$ there exists $L_\xi > 0$ such that operator $F_\xi(z)$ is $L_\xi$-Lipschitz, i.e., for~all $z_1, z_2\in \R^d:$
	\begin{equation*}
		\|F_\xi(z_1) - F_\xi(z_2)\| \le L_\xi\|z_1 - z_2\|.
	\end{equation*}
\end{assumption}
This assumption is widely used in the literature on variational inequalities problems.\\
The next assumption can be considered as a relaxation of standard strong monotonicity allowing $F_\xi(z)$ to be non-monotone with a certain structure.

\begin{assumption}[Strong monotonicity]\label{as:str_monotonicity}
	We assume that for~all $\xi$ operator $F_\xi(z)$ is $(\mu_{\xi},z^*)$-strongly monotone, i.e., there exists \textbf{(possibly negative)} $\mu_\xi \in \R$ such that for~all $z\in \R^d:$
	\begin{equation*}
		\langle F_\xi(z) - F_\xi(z^*), z - z^*\rangle \ge \mu_\xi\|z - z^*\|^2,
	\end{equation*}
 where $z^*$ is the solution of~\eqref{eq:variational-ineq-form}.
\end{assumption}

We emphasize that in some cases $\mu_{\xi}$ might be negative, which can lead to operator $F_\xi$ being non-monotone.

\subsection{Stochastic gradient-based algorithm with bias}

To analyze the convergence of \algname{SEG}, we consider a family of methods
\begin{align}
z^{k+\frac{1}{2}} &= z^k - \gamma_{1, \xi^k} \mathbf{g_{\xi^k}^{k}}, \nonumber \\
z^{k+1} &= z^k - \gamma_{2, \xi^k} \mathbf{g_{\xi^k}^{k+\frac{1}{2}}}, \label{eq:general_seg_method}
\end{align}
where $\gg_{\xi^k}^k := \gg_{\xi^k}(z^k)$ is some stochastic operator evaluated at point $z^k$ and $\xi^k$ encodes the~randomness~/~stochasticity appearing at iteration $k$ (e.g., it can be the sample used at step $k$). The choice of stepsizes varies across previous articles on this topic~\citep{mishchenko2020revisiting}. In our work we propose selecting stepsizes $\gamma_{1, \xi^k}, \gamma_{2, \xi^k}$ so that they will satisfy $0 < \gamma_{2, \xi^k} = \alpha \gamma_{1, \xi^k}$, where $0 < \alpha < 1$. The stepsizes are also allowed to depend on the same sample $\xi^k$.\\
Applying the fact, that we now have a biased oracle (Definition \ref{def:b_oracle}), we substitute $ \mathbf{g_{\xi^k}^{k}}$ so that Same sample \algname{SEG} will be presented in the following form:
\begin{align}
    z^{k+\frac{1}{2}} = z^k - \gamma_{1,\xi^k} F_{\xi^k}(z^k) - \gamma_{1,\xi^k} b(z^k), \nonumber \\
    z^{k+1} = z^k - \gamma_{2,\xi^k} F_{\xi^k}(z^{k+\frac{1}{2}}) - \gamma_{2,\xi^k} b(z^{k+\frac{1}{2}}),   \tag{S-SEG} \label{eq:S_SEG}
\end{align}
where $b(z^k): \R^d \rightarrow \mathbb{R}^d$ is a function, with a norm, bounded with $\zeta$ (see Definition \ref{def:b_oracle}). In each iteration the same sample $\xi^k$ is used for the exploration (computation of $z^{k+\frac{1}{2}}$) and update (computation of $z^{k+1}$) steps.\\
Let us now introduce the main assumptions on operator $\gg_{\xi}(z)$ (Definition \ref{def:b_oracle}) and stepsizes from \ref{eq:S_SEG} update rule which are the extensions of \cite[Assumption 2.1]{gorbunov2022stochastic}.

\begin{assumption}\label{as:unified_assumption_general}
	We assume that there exist non-negative constants $A, C, D_1, D_2, E_1, E_2, H_1, H_2, J, \zeta \geq~0,$ $ \rho \in [0; 1]$, and \textbf{(possibly random)} non-negative sequence $\{G_k\}_{k\ge 0}$ such that :
	\begin{equation}
	\label{eq:second_moment_bound_general} \Exp_{\xi^k}\left[\gamma_{2,\xi^k}^2\|\mathbf{g^{k+\frac{1}{2}}}\|^2\right]\leq 2A P_{k} + C\|z^k - z^*\|^2 + D_1 + E_1\zeta^2 - H_1 \zeta R,
	\end{equation}
		\begin{equation}
		\label{eq:P_k_general}
 P_{k} \geq \rho\|z^k - z^*\|^2 + JG_k - D_2 - E_2\zeta^2 + H_2 \zeta R,
	\end{equation}
	where $\mathbf{g^{k+\frac{1}{2}}} := \mathbf{g}_{\xi^k}(z^{k+\frac{1}{2}})$ is stochastic operator~\eqref{eq:general_seg_method}, $\gamma_{2,\xi^k}$ is a stepsize~(from \ref{eq:S_SEG} update rule),\\$z^*$ is the solution of~\eqref{eq:variational-ineq-form}, $R = \max_{i\in \{1..n\}} \| z^i - z^*\|$ and $P_{k} = \Exp_{\xi^k}\left[\gamma_{2,\xi^k}\langle \mathbf{g^{k+\frac{1}{2}}}, z^k - z^* \rangle\right]$.
\end{assumption}

Although those inequalities may seem unnatural, they are satisfied with certain parameters for several variants of \algname{S-SEG} under reasonable assumptions on the problem and the stochastic noise.\\
Finally, we need to imply the last assumption.
\begin{assumption}\label{as:stepsize_and_mu_conditions}
Operator $F_{\xi^k}$, stepsize $\gamma_{1,\xi^k}$ (from \ref{eq:S_SEG} update rule) and $\mu_{\xi^k}$ (from Assumption \ref{as:str_monotonicity}) satisfy the following conditions:
    \begin{equation}
        \Exp_{\xi^k}[\gamma_{1,\xi^k}F_{\xi^k}(z^*)] = 0, \label{eq:S_SEG_AS_stepsize_1}
    \end{equation}
    \begin{equation}
        \Exp_{\xi^k}[\gamma_{1,\xi^k}\mu_{\xi^k}(\obf_{\{\mu_{\xi^k} \ge 0\}} + 4\cdot\obf_{\{\mu_{\xi^k} < 0\}})] \ge 0, \label{eq:S_SEG_AS_stepsize_and_mu}
    \end{equation}
where $z^*$ is the solution of~\eqref{eq:variational-ineq-form}.
\end{assumption}

Equation~\ref{eq:S_SEG_AS_stepsize_1} is a generalization of unbiasedness at $z^*$, since $F(z^*) = 0$, and the left-hand side of \eqref{eq:S_SEG_AS_stepsize_and_mu} is a generalization of the averaged quasi-strong monotonicity constant multiplied by the stepsize. Moreover, \eqref{eq:S_SEG_AS_stepsize_and_mu} holds when all $\mu_\xi \ge 0$, which is typically assumed in the analysis of \algname{S-SEG}.\\\\
We want to present some specific distributions in \algname{S-SEG} algorithm, and for starters we assume that $F(z) = \frac{1}{n}\sum_{i=1}^n F_i(z)$ and $F_i(z)$ is $(\mu_i, z^*)$-strongly monotone and $L_i$-Lipschitz.

\begin{example}[Uniform Sampling, example from \citep{gorbunov2022stochastic}]\label{ex:uniform_sampling}
    Let $\xi^k$ be sampled from the uniform distribution $U[0; n]$, i.e., for all $i\in \{1..n\}$ we have $\Prob{\xi^k = i} = p_i \equiv \frac{1}{n}$. If
    \begin{equation}\label{eq:average_mu}
        \overline{\mu} = \frac{1}{n}\sum\limits_{i: \mu_i \geq 0} \mu_i + \frac{4}{n}\sum\limits_{i: \mu_i < 0} \mu_i \ge 0
    \end{equation}
    and $\gamma_{1,\xi} \equiv \gamma > 0$, then Assumption~\ref{as:stepsize_and_mu_conditions} holds.
\end{example}

In the example above the oracle is unbiased and, as the result, we use constant stepsize $\gamma_{1,\xi} = \gamma$. This leads us to our goal - to specify the mentioned example to a biased oracle. Moreover, we emphasize that to fulfill Assumption~\ref{as:stepsize_and_mu_conditions} in Example~\ref{ex:uniform_sampling} we only need to assume that parameter $\gamma$ is positive.
We also provide an example of a non-uniform sampling.
\begin{example}[Importance Sampling, example from \citep{gorbunov2022stochastic}]\label{ex:importance_sampling}
    Let $\xi^k$ be sampled from the following distribution:\\for all $i\in \{1..n\}$
    \begin{equation}
        \Prob{\xi^k = i} = p_i = \frac{L_i}{\sum\limits_{j=1}^n L_j}. \label{eq:importance_sampling_prob_main}
    \end{equation}
    If \eqref{eq:average_mu} is satisfied and $\gamma_{1,\xi} = \frac{\gamma \overline{L}}{L_\xi}$, where $\overline{L} = \frac{1}{n}\sum_{i=1}^n L_i$, $\gamma > 0$, then Assumption~\ref{as:stepsize_and_mu_conditions} holds.
\end{example}

\subsection{Convergence rate of S-SEG with biased oracle}

Then we need to undertake some preparatory steps to derive the convergence rate of \algname{S-SEG} algorithm with a biased oracle. If we assume, that Assumptions~\ref{as:lipschitzness},~\ref{as:str_monotonicity}~and~\ref{as:stepsize_and_mu_conditions}~hold, and that stepsize $\gamma_{1,\xi^k}$ from \algname{S-SEG} update rule satisfies the following inequality:
\begin{equation}
	    \gamma_{1,\xi^k} \leq \frac{1}{4|\mu_{\xi^k}| + 2L_{\xi^k}},\label{eq:S_SEG_AS_stepsize_upper_bound}
\end{equation}
then we are able to show that  Assumption~\ref{as:unified_assumption_general} holds (see Appendix for this derivation) for \algname{S-SEG}. Now we can explicitly write down unknown variables from Assumption~\ref{as:unified_assumption_general} and derive convergence guarantees for \algname{S-SEG} with biased oracle.

\begin{theorem}\label{thm:S-SEG_convergence_rate}Let Assumptions~\ref{as:lipschitzness}~and~\ref{as:str_monotonicity} hold. If stepsizes from \ref{eq:S_SEG} update rules satisfy $\gamma_{2,\xi^k} = \alpha \gamma_{1,\xi^k}$, where $\alpha > 0$, and $\gamma_{1,\xi^k}$ satisfies Assumption~\ref{as:stepsize_and_mu_conditions} and \eqref{eq:S_SEG_AS_stepsize_upper_bound}, then $\mathbf{g^{k+\frac{1}{2}}}$ from~\ref{eq:S_SEG} satisfies Assumption~\ref{as:unified_assumption_general}. \\\\
If additionally $\alpha \leq \frac{1}{8}$ and  $ \rho > 0$ ($\rho$ is defined in Assumption \ref{as:unified_assumption_general}), then for all $N > 0$ (number of steps for \ref{eq:S_SEG} algorithm with the output $z^{N}$): 
	\begin{equation*}
		\Exp\left[\|z^{N+1} - z^*\|^2\right] \leq \left(1 - \rho\right)\Exp\left[\|z^N - z^*\|^2\right] + \frac{5\alpha(8\alpha + 1)}{2}\sigma_{\algname{AS}}^2 + \frac{\alpha(20\alpha + 1)}{2} \widehat{\gamma_2} \zeta^2 - \alpha(8\alpha + 1) \widehat{\gamma_1}\zeta R,
	\end{equation*}
where $z^*$ is the solution of~\eqref{eq:variational-ineq-form}, $\sigma_{\algname{AS}}^2 = \Exp_\xi\left[\gamma_{1, \xi}^2\|F_\xi(z^*)\|^2\right]$, $\widehat{\gamma_1} = \Exp_\xi\left[\gamma_{1,\xi^k}\right] $ and $\widehat{\gamma_2} = \Exp_\xi\left[\gamma_{2,\xi^k}\right]$,\\$\zeta > 0$ is the upper-bound for bias norm (see Definition \ref{def:b_oracle}) and $R =\max_{i\in \{0..N\}}\|z^i - z^*\|.$ \\\\
The proof of Theorem \ref{thm:S-SEG_convergence_rate} can be found in supplementary materials (Appendix~\ref{proof:S-SEG_convergence_rate_proof}).
 
\end{theorem}

It is noticeable, that this result differs from the case of unbiased oracle only with coefficient in the middle term and with addition of the last two terms, with quadratic and linear dependence from the noise level accordingly.

\subsection{Convergence rates comparison in Uniform Sampling setup and Importance Sampling setup}

The result obtained above can be specified on the Uniform Sampling (see Example~\ref{ex:uniform_sampling}) when we substitute abstract distribution $\mathcal{D}$ for $U[0; n]$ in \eqref{eq:variational-ineq-form} and as a direct corollary of Theorem~\ref{thm:S-SEG_convergence_rate} we can easily derive convergence rate for diminishing stepsize.

\begin{corollary}[\algname{S-SEG} with Uniform Sampling setup]\label{cor:S-SEGUS-o-convergence}
Consider the setup from Example~\ref{ex:uniform_sampling} with $\overline{\mu} > 0$. Let stepsizes from \ref{eq:S_SEG} update rules satisfy $\gamma_{2,\xi^k} = \alpha \gamma_{1,\xi^k}$, where $\alpha = \frac{1}{8}$, and stepsize $\gamma_{1,\xi^k} = \beta_k\gamma = \frac{\beta_k}{6L_{\max}}$, where $L_{\max} = \max_{i\in \{1..n\}}L_i$ and $0 < \beta_k \leq 1$.\\Then for all $N > 0$ (number of steps for \ref{eq:S_SEG} algorithm with the output $z^{N}$):
        \begin{equation*}
	    \Exp\left[\|z^N - z^*\|^2\right] = \mathcal{O}\left( \frac{L_{\max}R_0^2}{\overline{\mu}}\exp \left( -\frac{\overline{\mu} N}{L_{\max}} \right) + \frac{\sigma_{\algname{US}*}^2}{\overline{\mu}^2 N} + \frac{\zeta^2}{\overline{\mu}^2}\right),
	\end{equation*} 
where $z^*$ is the solution of~\eqref{eq:variational-ineq-form}, $R_0^2 = \|z^0 - z^*\|^2, \sigma_{\algname{US}*}^2 = \frac{1}{n}\sum_{i=1}^n \|F_i(z^*)\|^2$ , $\overline{\mu}$ is defined in Example~\ref{ex:uniform_sampling}\\and $\zeta > 0$ is the upper-bound for bias norm (see Definition \ref{def:b_oracle}).\\\\
The proof of Corollary \ref{cor:S-SEGUS-o-convergence} can be found in supplementary materials (Appendix~\ref{cor:S-SEGUS-o-convergence-proof}).

\end{corollary}

We can also provide an example of \algname{S-SEG} usage on non-uniform distribution. Let us consider Importance Sampling (see Example~\ref{ex:importance_sampling}) with probabilities depending on the Lipschitz constants. The next corollary will show the convergence rate in this particular case, as well as underline the difference between the usage of two mentioned distributions.

\begin{corollary}[\algname{S-SEG} with Importance Sampling setup]\label{cor:S-SEGIS-o-convergence}
Consider the setup from Example~\ref{ex:importance_sampling} with $\overline{\mu} > 0$. Let stepsizes from \ref{eq:S_SEG} update rules satisfy\\ $\gamma_{2,\xi^k} = \alpha \gamma_{1,\xi^k}$, $\alpha = \frac{1}{8}$, and $\gamma_{1,\xi^k} = \frac{\beta_k\gamma \overline{L}}{L_{\xi^k}} = \frac{\beta_k}{6L_{\xi^k}}$, where $\overline{L} = \frac{1}{n}\sum_{i=1}^n L_i$ and $0 < \beta_k \leq 1$.\\Then for all $N > 0$ (number of steps for \ref{eq:S_SEG} algorithm with the output $z^{N}$):
        \begin{equation*}
	\Exp\left[\|z^N - z^*\|^2\right] = \mathcal{O}\left( \frac{\overline{L}R_0^2}{\overline{\mu}}\exp \left( -\frac{\overline{\mu} N}{\overline{L}} \right) + \frac{\sigma_{\algname{IS}*}^2}{\overline{\mu}^2 N} + \frac{\zeta^2}{\overline{\mu}^2}\right),
	\end{equation*} 
where $z^*$ is the solution of~\eqref{eq:variational-ineq-form}, $R_0^2 = \|z^0 - z^*\|^2, \sigma_{\algname{IS}*}^2 = \frac{1}{n}\sum_{i=1}^n \frac{\overline{L}}{L_i}\|F_i(z^*)\|^2$ , $\overline{\mu}$ is defined in Example~\ref{ex:importance_sampling}\\and $\zeta > 0$ is the upper-bound for bias norm (see Definition \ref{def:b_oracle}).\\\\
The proof of Corollary \ref{cor:S-SEGIS-o-convergence} can be found in supplementary materials (Appendix Corollary~\ref{cor:S-SEGIS-o-convergence-proof}).
\end{corollary}

It is worth mentioning, that unlike the rate for \algname{S-SEG-US} (\algname{S-SEG} with Uniform Sampling, Corollary~\ref{cor:S-SEGUS-o-convergence}), the rate obtained in Corollary~\ref{cor:S-SEGIS-o-convergence} is influenced by the average Lipschitz constant $\overline{L}$, which can be significantly lower than the largest constant $L_{\max}$. This leads to decreasing the exponential term for convergence rate in \algname{S-SEG-IS} (\algname{S-SEG} with Importance Sampling) compared to that of \algname{S-SEG-US}. 
Moreover, the theoretical foundations of \algname{S-SEG-IS} allow for the adoption of a significantly elevated value of $\gamma$.\\
Additionally, a greater magnitude of $F_i(x^*)$ suggests an increased $L_i$, exemplified by the relationship $|F_i(x^*)|^2 \approx L_i^2$. In such instances, $\sigma_{\algname{IS}}^2$ is approximately proportional to $\overline{L}^2$, while $\sigma_{\algname{US}}^2$ aligns with $\overline{L^2} = \frac{1}{n}\sum_{i=1}^n L_i^2$, which is greater than or equal to $\overline{L}^2$, thus leading to lower estimation for second term for convergence in \algname{S-SEG-IS} compared to that of \algname{S-SEG-US}.\\
In the subsequent section, we will extrapolate these convergence rates to the context of gradient approximation employing a gradient-free oracle, thereby deriving more specialized outcomes. Nonetheless, it is pertinent to acknowledge that the discussions in this section, being predicated on the unspecified operator $F$, render these findings applicable to a substantially broader spectrum of problems.

\section{Gradient-free algorithm for SPP} \label{sec:Main_Result}
In this section, we present the main result of this work, namely a gradient-free algorithm for solving a stochastic SPP in an overparameterized setup. We study it in the following form:
\begin{equation}
    \label{eq:main-spp-problem}
    \min_{x \in \mathbb{R}^{d_x}}\max_{y \in \mathbb{R}^{d_y}}\left\{ f(x, y) := \expect{f(x, y, \xi)} \right\},
\end{equation}
where $f: \mathbb{R}^d \rightarrow \mathbb{R},\quad d := d_x + d_y$ is $L$-smooth convex-concave function. Our approach to develop gradient-free algorithm is based on the biased \algname{S-SEG-US} algorithm described in the previous section. Instead of a biased gradient oracle we use the gradient~approximation with certain assumptions, so let us introduce them together with the main definitions.

\subsection{Definitions and assumptions on the objective function}
Instead of stochastic operator $F_\xi$, like in the section above, we will now use gradient approximation of a gradient-free oracle of our objective function. Firstly, we need to introduce gradient-free oracle.
\begin{definition}[Gradient-free oracle]\label{def:gradientfreeoracle}

Gradient-free oracle is defined via adding bounded bias to objective function $f$:
    \begin{equation*}
    \Tilde{f}(z, \xi) = f(z, \xi) + \delta(z),
    \end{equation*}
where $\delta(z)$ is such that there exists $\Delta > 0$: for all $z \in \R^d$ $|\delta(z)| \leq \Delta .$ 
\end{definition}

It is noticeable, that since we have $l_2$-norm, if for all $z \in \R^d$ $|\delta(z)| = 0,$ then the gradient-free oracle value equals the true value of function $f(z, \xi)$.\\
We will use a gradient~approximation via $l_2$ randomization in a following form:
\begin{definition}[Gradient approximation]\label{def:gradientfreeapprox}
\begin{equation*}
        g(z,\xi,e) = \frac{d}{2 \tau} \left( \Tilde{f}(z + \tau e, \xi) - \Tilde{f}(z - \tau e, \xi) \right)e,
    \end{equation*}
where $e :=  \left(e_x, e_y \right)^T$ is a vector uniformly distributed on unit sphere $S^d(1)$ and $\tau$ is a smoothing parameter.
\end{definition}

Then we assume that gradient noise is bounded as follows.
\begin{assumption} \label{ass:stoch_noise}
    There exists $\sigma^2_* \geq 0$ such that for all $z \in \mathbb{R}^d:$
    \begin{equation*}
        \expect{\norms{\nabla f(z^*, \xi) - \nabla f(z^*)}^2} \leq \sigma^2_*.
    \end{equation*}
\end{assumption}

In this section we also need to assume strong-convexity and $L$-smoothness of the function $f(z, \xi)$. However, it is worth noting that these assumptions are just subcases of more general Assumptions \ref{as:lipschitzness} and \ref{as:str_monotonicity}, for operator $F_{\xi}(z) := \left[ \nabla_x f(x, y, \xi); -\nabla_y f(x, y, \xi)\right]$, so we can refer to them. In the following part of the article we will refer to some parameters like $L_{\max} = \max_{i\in \{1..n\}}L_i$ or $\overline{\mu}$ that are derived from Assumptions \ref{as:lipschitzness} and \ref{as:str_monotonicity} and were defined in the section above.

\subsection{ZOS-SEG algorithm}
In this subsection we present a gradient-free algorithm for solving a stochastic SPP~\eqref{eq:main-spp-problem}  in an overparametrized setup. Our approach is based on the biased \algname{S-SEG} algrorithm. Instead of gradient oracle we use gradient approximation via $l_2$ randomization (Definition \ref{def:gradientfreeapprox}).

\begin{algorithm}[H]
    \caption{Zero-Order Same-sample Stochastic Extragradient algorithm (\algname{ZOS-SEG})}
    \label{alg:zosseg}
        \textbf{Input}: Starting point $z^0 \in \R^d$, stepsize $\gamma_1 > 0$ satisfies \eqref{eq:S_SEG_AS_stepsize_upper_bound}, stepsize $\gamma_2 := \alpha \gamma_1, \alpha = \frac{1}{8},$\\number of iterations $N \in \mathbb{Z}_+$, parameter $\tau > 0$, batch size $B \in \mathbb{Z}_+$. 
        \hspace*{\algorithmicindent}
            \begin{algorithmic}[1] 
            	\For{$k = 0, 1, 2, \ldots, N - 1$}
            	\State Sample uniformly distributed on $S^d(1)$ vectors $\{e_1,...,e_{B}\}$ and i.i.d $\xi_1,...,\xi_{B} \sim \mathcal{D} $
                    \State Define $g_k^B = \frac{1}{B} \sum_{i=1}^{B} g(z^k, \xi_i, e_i)$ using Definition \ref{def:gradientfreeapprox}
            	\State $z^{k + \frac{1}{2}} = z^k - \gamma_1 g_k^B$
                    \State Define $g_{k + \frac{1}{2}}^B = \frac{1}{B} \sum_{i=1}^{B} g(z^{k + \frac{1}{2}}, \xi_i, e_i)$ using Definition \ref{def:gradientfreeapprox}
            	\State $z^{k + 1} = z^k - \gamma_2 g_{k + \frac{1}{2}}^B$
            	\EndFor
            \end{algorithmic}
        \textbf{Output}: $z^N$
\end{algorithm}
It is obvious that gradient approximation satisfies the form of biased oracle (see Definition~\ref{def:b_oracle}), so we can simply reduce the estimation of the convegence of proposed algorithm to the case of biased oracle by substituting upper bounds on the bias and second moment for the gradient approximation (see Appendix for these derivations) in the convergence of \algname{S-SEG} algorithm. We also use batch-size $B$ in the algorithm, let us formally define a batched biased gradient approximation:
    \begin{equation*}
        g^B(z, \xi, e) := \frac{1}{B} \sum_{i=1}^{B} g(z, \xi_i, e_i) \quad \text{for i.i.d. } \xi_1, \xi_2, ..., \xi_B \sim  U[0; n], 
    \end{equation*} \\
where $e = \{e_1,...,e_{B}\}$ are uniformly distributed on $S^d(1)$ vectors.
Then we are able to derive convergence rate of our algorithm.

\subsection{Convergence in terms of the deterministic noise}\label{subsec:determ_noise}

\begin{lemma} \label{lem:zosseg-convergence-rate}
    Consider the setup from Example~\ref{ex:uniform_sampling}. Let Assumptions \ref{as:lipschitzness} and \ref{as:str_monotonicity} hold and let gradient approximation (Definition~\ref{def:gradientfreeapprox}) satisfy Assumption \ref{ass:stoch_noise}. Then for the number of iterations $N$, batch size $B$ and the output $z^N = (x^N,y^N)$ Zero-Order~Same-sample~Stochastic~Extragradient algorithm with Uniform Sampling ($\mathcal{D} \sim U[0; n]$, see Algorithm~\ref{alg:zosseg}) converges with the following rate: 
    \begin{eqnarray*}
       \expect{ \| z^N - z^* \| ^ 2} 
        & \lesssim \frac{L_{\max}R_0^2}{\overline{\mu}}\exp\left( -\frac{\overline{\mu}N}{L_{\max}} \right)  + \frac{ d \sigma_{*}^2}{\overline{\mu}^2 N B} + \frac{d L_{\max}^2 \tau^2}{\overline{\mu}^2 N B} + \frac{d^2 \Delta^2}{\overline{\mu}^2 N B \tau^2} + \frac{L_{\max}^2 \tau^2}{\overline{\mu}^2 } +  \frac{d^2 \Delta^2}{\overline{\mu}^2 \tau^2} ,
    \end{eqnarray*}
   where $L_{\max} = \max_{i\in \{1..n\}}L_i$, $R_0^2 = \|z^0 - z^*\|^2$, $\Delta$ is the upper-bound for noise norm (see Definition \ref{def:gradientfreeoracle}),\\ $\tau$ is the smoothing parameter (see Definition \ref{def:gradientfreeapprox}) and d is the dimension of space.\\\\
   The proof of Lemma \ref{lem:zosseg-convergence-rate} can be found in supplementary materials (Appendix Lemma~\ref{proof:zosseg-convergence-rate-lemma}).
\end{lemma}
 
We then can proceed to estimating the unknown variables. In order to do so, we will upper-bound $\expect{ \| z^N - z^* \| ^ 2}$ with $\varepsilon$ and derive the main result for this section.

  \begin{theorem}
      \label{th:zosseg-epsilon-acc}
        Consider the setup from Example~\ref{ex:uniform_sampling}. Let Assumptions \ref{as:lipschitzness} and \ref{as:str_monotonicity} hold and\\let gradient approximation (Definition~\ref{def:gradientfreeapprox}) satisfy Assumption \ref{ass:stoch_noise}. Let $L_{\max} = \max_{i\in \{1..n\}}L_i$, $R_0^2 = \|z^0 - z^*\|^2$.\\Let the smoothing parameter $\tau$ satisfy $\tau \leq  \frac{\overline{\mu}}{L_{\max}} \sqrt{\varepsilon}$, adversarial noise $\Delta$ and batch size $B$ satisfy:
        \begin{eqnarray*}
        \Delta &\leq& \frac{\varepsilon \overline{\mu}^2}{L_{\max} \sqrt{d}} \min \left\{ \max \left\{ 1 , \frac{ \sigma_{*}}{\overline{\mu} \sqrt{\varepsilon}}, \sqrt{\frac{L_{\max} \ln \left( \frac{R_0^2 L_{\max}}{\varepsilon \overline{\mu}} \right)}{d \overline{\mu}}} \right\}, \frac{1}{\sqrt{d}} \right\},\\
        B &=& \mathcal{O}\left( \max\left\{ \frac{d \overline{\mu}}{L_{\max} \ln \left( \frac{R_0^2 L_{\max}}{\varepsilon \overline{\mu}} \right)} \max \left\{ \frac{\sigma_{*}^2}{\varepsilon \overline{\mu}^2 } ,  1 \right\}, 1 \right\} \right).
        \end{eqnarray*}
        Then Zero-Order~Same-sample~Stochastic~Extragradient algorithm~with Uniform Sampling ($\mathcal{D} \sim U[0, n]$, see Algorithm~\ref{alg:zosseg})
        achieves $\varepsilon$-accuracy: $\expect{ \| z^N - z^* \| ^ 2}~\leq~ \varepsilon$ with number of iterations $N$ being: 
        \begin{eqnarray*}
           N = \mathcal{O}\left( \frac{L_{\max}}{\overline{\mu}} \ln \left( \frac{R_0^2 L_{\max}}{\varepsilon \overline{\mu}} \right) \right).
        \end{eqnarray*}
        The total number of calls to the gradient-free oracle $T$ is:
        \begin{eqnarray*}
             T = N \cdot B = \mathcal{O}\left(\max \left\{ d, \frac{d \sigma_{*}^2}{\varepsilon \overline{\mu}^2}, \frac{L_{\max}}{\overline{\mu}} \ln \left( \frac{R_0^2 L_{\max}}{\varepsilon \overline{\mu}} \right) \right\} \right).
        \end{eqnarray*}
    \end{theorem}

\textit{Sketch of the proof (for full derivation see Appendix Theorem~\ref{proof:zosseg-epsilon-acc-proof}):}\\
We upper-bound each term with $\varepsilon$, and from that, we are able to get explicit estimations for $N$ and $\tau$, while estimations for $\Delta$ and $B$ are derived from several inequalities.\\
The outcome of Theorem \ref{th:zosseg-epsilon-acc} demonstrates the effective iterative complexity $N$. It is important to mention that both the batch size $B$ and the maximum admissible level of adversarial noise $\Delta$ can vary over time. This means they directly depend on $\sigma^2_*$, resulting in an optimal estimate for the number of gradient-free oracle calls $T$. Furthermore, it should be noted that we do not claim that this assessment is definitive. We propose an open question for future research: enhancing the estimate of the maximum admissible noise level, as well as identifying upper bounds for the noise level, beyond which convergence in the overparameterized setup cannot be ensured.

\subsubsection{Comparison in Uniform Sampling setup and Importance Sampling setup}

If we want to use Importance Sampling (see Example~\ref{ex:importance_sampling}) in Algorithm~\ref{alg:zosseg} we can notice that due to the results of Corollaries~\ref{cor:S-SEGUS-o-convergence}~and~\ref{cor:S-SEGIS-o-convergence}, which derive the same estimation of convergence rate, except that $L_{\max}$ is replaced with $\overline{L}$, which in most scenarios leads to better estimation.\\\\
The most significant distinction between the Uniform Sampling setup and Importance Sampling setup convergence rates is the fact that the estimation for Importance Sampling setup will be lower due to the exponential term and terms that are quadratically dependent on $L_{\max}$, as in most cases $\overline{L}$ is less than $L_{\max}$. Consequently, the per-item estimates for number of iterations $N$ and total number of calls $T$ will also be better. On the other hand, we impair the estimates for adversarial noise $\Delta$ and batch size $B$ due to the inverse dependence on $L_{\max}$, which may affect the overall reliability of the statistical model or computational algorithm being implemented.

\subsection{ZOS-SEG algorithm and convergence in terms of the stohastic noise}\label{subsec:stoch_noise}

We then substitute the deterministic noise in the gradient-free oracle (see Definition~\ref{def:gradientfreeoracle}) to a stochastic noise with bounded second moment. This will lead to a slight change in the convergence rate of Algorithm~\ref{alg:zosseg} and the estimation on the adversarial noise $\Delta$. However, other estimations on number of iterations $N$, batch size $B$ and the smoothing parameter $\tau$ will remain exactly the same. Let us define a gradient-free oracle with stochastic noise.

\begin{definition}[Gradient-free oracle with stochastic noise]\label{def:gradientfreeoracle-stochnoise}
Gradient-free oracle with stochastic noise is defined via adding random variable with bounded second moment to the true value of the objective function $f$:
    \begin{equation*}
    \Tilde{f}(z, \xi) = f(z) + \xi,
    \end{equation*}
where $\xi$ is a random variable that there exists $\Delta > 0$: $\expect \xi^2 \leq \Delta^2.$ 
\end{definition}

The definition of gradient approximation will also change, due to the fact, that query points must have different realizations of independent random variables from the same distribution.
\begin{definition}[Gradient approximation with stochastic noise]\label{def:gradientfreeapprox-stochnoise}
\begin{equation*}
        g(z,\xi, \xi', e) = \frac{d}{2 \tau} \left( \Tilde{f}(z + \tau e, \xi) - \Tilde{f}(z - \tau e, \xi') \right)e,
    \end{equation*}
where $\Tilde{f}$ is gradient-free oracle with stochastic noise (see Definition~\ref{def:gradientfreeoracle-stochnoise}), $\xi$ and $\xi'$ are i.i.d, $e :=  \left(e_x, e_y \right)^T$ is a vector uniformly distributed on unit sphere $S^d(1)$ and $\tau$ is a smoothing parameter.
\end{definition}
This realization of gradient approximation differs from Definition~\ref{def:gradientfreeapprox} with dependence on two i.i.d random variables. The implementation of Algorithm~\ref{alg:zosseg} will remain almost the same, but on each iteration we will need to sample i.i.d $\xi_1, \xi'_1,..., \xi_B, \xi'_B \sim \mathcal{D}$. Then $g_k^B = \frac{1}{B} \sum_{i=1}^{B} g(z^k, \xi_i, \xi'_i, e_i)$ will be defined using Definition \ref{def:gradientfreeapprox-stochnoise}.\\
Using the results obtained above we can estimate convergence rate in case of the gradient approximation with stochastic noise. The proof of the decent lemmas is in the Appendix. The main difference from the proof for deterministic noise is considering the fact, that $\xi$ and $\xi'$ are i.i.d and also are independent with vector $e$, uniformly distributed on unit sphere $S^d(1)$ (see Definition~\ref{def:gradientfreeapprox-stochnoise}).

\begin{lemma} \label{lem:zosseg-convergence-rate-stoch-noise}
    Consider the setup from Example~\ref{ex:uniform_sampling}. Let Assumptions \ref{as:lipschitzness} and \ref{as:str_monotonicity} hold and let gradient approximation (Definition~\ref{def:gradientfreeapprox-stochnoise}) satisfy Assumption \ref{ass:stoch_noise}. Then for the number of iterations $N$, batch size $B$ and the output $z^N = (x^N,y^N)$ Zero-Order~Same-sample~Stochastic~Extragradient algorithm with Uniform~Sampling (see Algorithm~\ref{alg:zosseg}) converges with the following rate: 
    \begin{eqnarray*}
       \expect{ \| z^N - z^* \| ^ 2} 
        & \lesssim \frac{L_{\max}R_0^2}{\overline{\mu}}\exp\left( -\frac{\overline{\mu}N}{L_{\max}} \right)  + \frac{ d \sigma_{*}^2}{\overline{\mu}^2 N B} + \frac{d L_{\max}^2 \tau^2}{\overline{\mu}^2 N B} + \frac{d^2 \Delta^2}{\overline{\mu}^2 N B \tau^2} + \frac{L_{\max}^2 \tau^2}{\overline{\mu}^2 },
    \end{eqnarray*}
   where $L_{\max} = \max_{i\in \{1..n\}}L_i$, $R_0^2 = \|z^0 - z^*\|^2$, $\Delta$ is the adversarial noise (see Definition \ref{def:gradientfreeoracle-stochnoise}), $\tau$ is the smoothing parameter (see Definition \ref{def:gradientfreeapprox-stochnoise}) and d is the dimension of space.\\\\
   The proof of Lemma \ref{lem:zosseg-convergence-rate-stoch-noise} can be found in supplementary materials (Appendix Lemma~\ref{proof:zosseg-convergence-rate-lemma-stoch-noise}).
\end{lemma}

We can notice, that convergence rate for Algorithm~\ref{alg:zosseg}
with the gradient approximation with stochastic noise has equal terms as with the deterministic noise, except for the abundance of the last term $\frac{d^2 \Delta^2}{\overline{\mu}^2 \tau^2}$. Not only it reduces the total upper-bound on the convergence rate, but also implies less restrictions on the adversarial noise $\Delta$ as we can proof with the following theorem.

\begin{theorem}
  \label{th:zosseg-epsilon-acc-stochnoise}
    Consider the setup from Example~\ref{ex:uniform_sampling}. Let Assumptions \ref{as:lipschitzness} and \ref{as:str_monotonicity} hold and\\let gradient approximation (Definition~\ref{def:gradientfreeapprox-stochnoise}) satisfy Assumption \ref{ass:stoch_noise}. Let $L_{\max} = \max_{i\in \{1..n\}}L_i$, $R_0^2 = \|z^0 - z^*\|^2$.\\\\
    Let the smoothing parameter $\tau$ satisfy $\tau \leq \max \left\{ \frac{\sigma_{*}}{L_{\max}}, \sqrt{\frac{\varepsilon\overline{\mu} \ln \left( \frac{R_0^2 L_{\max}}{\varepsilon \overline{\mu}} \right)}{L_{\max} d }}\right\}$, adversarial noise $\Delta$ and batch size B satisfy:
    \begin{eqnarray*}
    \Delta &\leq& \frac{\sqrt{\varepsilon}  \sigma_{*}}{L_{\max} d} \max \left\{ 1 , \frac{\overline{\mu}^2 \sqrt{\varepsilon} \ln \left( \frac{R_0^2 L_{\max}}{\varepsilon \overline{\mu}} \right)}{ \sigma_{*}  \sqrt{d}}, \sqrt{\frac{ \overline{\mu}^3  \ln \left( \frac{R_0^2 L_{\max}}{\varepsilon \overline{\mu}} \right)}{L_{\max}}},     \sqrt{\frac{ \varepsilon \overline{\mu} L_{\max} \ln \left( \frac{R_0^2 L_{\max}}{\varepsilon \overline{\mu}} \right)}{d}} \right\},  \\
    B &=& \mathcal{O}\left( \max\left\{ \frac{d \sigma_{*}^2}{\varepsilon \overline{\mu} L_{\max} \ln \left( \frac{R_0^2 L_{\max}}{\varepsilon \overline{\mu}} \right)} , 1 \right\} \right).
    \end{eqnarray*}
    Then Zero-Order~Same-sample~Stochastic~Extragradient algorithm with Uniform~Sampling~(see Algorithm~\ref{alg:zosseg})
    achieves $\varepsilon$-accuracy: $\expect{ \| z^N - z^* \| ^ 2}~\leq~ \varepsilon$ with number of iterations $N$ being: 
    \begin{eqnarray*}
       N = \mathcal{O}\left( \frac{L_{\max}}{\overline{\mu}} \ln \left( \frac{R_0^2 L_{\max}}{\varepsilon \overline{\mu}} \right) \right).
    \end{eqnarray*}
    The total number of calls to the gradient-free oracle $T$ is:
    \begin{eqnarray*}
         T = N \cdot B = \mathcal{O}\left(\max \left\{ \frac{d \sigma_{*}^2}{\varepsilon \overline{\mu}^2}, \frac{L_{\max}}{\overline{\mu}} \ln \left( \frac{R_0^2 L_{\max}}{\varepsilon \overline{\mu}} \right) \right\} \right).
    \end{eqnarray*}
\end{theorem}
The proof of Theorem~\ref{th:zosseg-epsilon-acc-stochnoise} can be found in supplementary materials (Appendix Theorem~\ref{proof:zosseg-epsilon-acc-proof-stoch-noise}).\\\\
When contemplating the implementation of Importance Sampling, as illustrated in Example~\ref{ex:importance_sampling}, it becomes evident that, pursuant to the findings of Corollaries~\ref{cor:S-SEGUS-o-convergence} and \ref{cor:S-SEGIS-o-convergence}, an analogous convergence rate estimation is deduced. The distinction lies in the substitution of $L_{\max}$ with $\overline{L}$, which, in a majority of cases, culminates in a more favorable estimation.

\subsubsection{Comparison in deterministic noise setup and stochastic noise setup}

By replacing the deterministic noise with stochastic noise in the definition of a gradient-free oracle, we have achieved a smaller number of terms in the convergence rate estimate of our algorithm. However, all the parameter estimates remain exactly the same as in the deterministic case, except for the estimate of adversarial noise $\Delta$, which has increased due to the absence of one of the terms dependent on it and instead of taking the minimum of two arguments, we now only have one of them. Since we aim to reduce the noise, the case of stochastic noise is less advantageous.

\newpage
\section{Experiments} \label{sec:Experiments}
In this section, we will use a simple example to verify the theoretical results, namely to show the convergence of the proposed Zero-Order~Same-sample~Stochastic~Extragradient algorithm (\algname{ZOS-SEG}, see Algorithm \ref{alg:zosseg}). We used implementation using Uniform Sampling. The optimization problem \eqref{eq:main-spp-problem} is as follows:
\begin{eqnarray}\label{eq:experiments_problems}
    \min_{x \in \mathbb{R}^{d_x}} \max_{y \in \mathbb{R}^{d_y}}f(x,y) = \min_{x \in \mathbb{R}^{d_x}} \max_{y \in \mathbb{R}^{d_y}} \frac{1}{n} \sum_{i=1}^n f_i(x, y),&&\\
    f_i(x, y) := x^T C_i y + \frac{\lambda_i}{2} ||x||^2 - \frac{\lambda_i}{2} ||y||^2,\nonumber
\end{eqnarray} 
where $C_i \in \mathbb{R}^{d_x \times d_y}$ for all $1 \leq i \leq n$, $\lambda_i$ is a regularization paramether and overparameterization condition holds ($d > n$). Problem \eqref{eq:experiments_problems} is a convex-concave stochastic optimization problem \eqref{eq:main-spp-problem}.\\
We conducted experiments for different values of parameter $\Delta$ (see Definitions~\ref{def:gradientfreeoracle}~and~~\ref{def:gradientfreeoracle-stochnoise}). For the deterministic noise we use the following function: $\delta(z) = \frac{\Delta}{1 + \|z\|}$. For the stochastic noise we use Standard Normal distribution multiplied by $\Delta$.
We optimize $f(x, y)$ \eqref{eq:experiments_problems} with parameters: $d_x = 64, d_y = 64$ (dimensional of problem), $n = 32$
(number of functions in a finite sum), $\tau = 1$
(smoothing parameter), $\gamma = 0.05$ (step size),
$B = 128$ (batch size). In the Appendix subsection~\ref{app:experiments} you can see other graphs for algorithm \algname{ZOS-SEG} with different values of batch size $B$.\\

\begin{figure}[htbp]
\begin{minipage}[h]{\linewidth}
\centering
\includegraphics[width=0.7\linewidth, height=9.5cm]{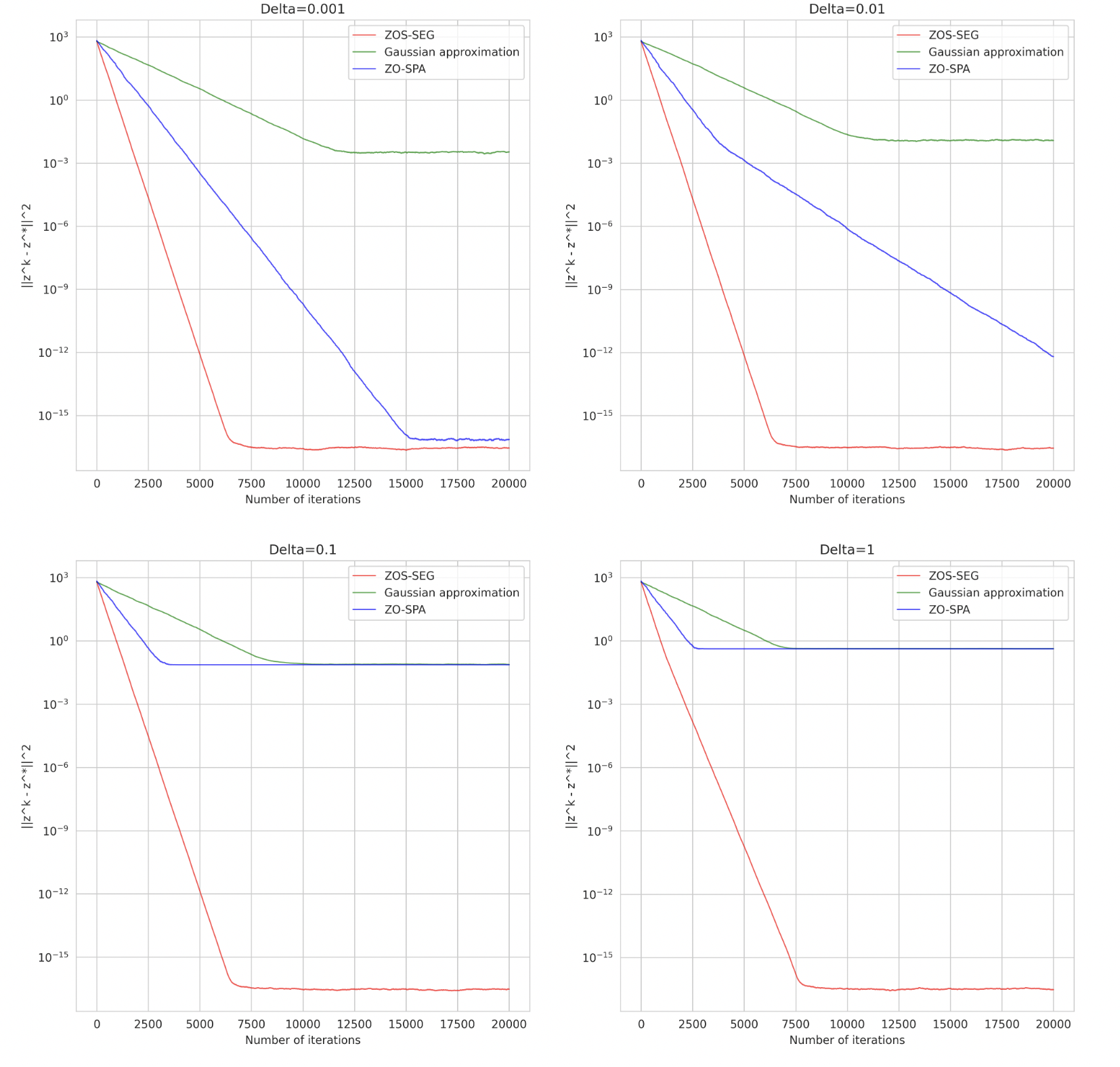}
\captionsetup{justification=centering}
\caption{Convergence of the Zero-Order~Same-sample~Stochastic~Extragradient (\algname{ZOS-SEG}) algorithm, Zeroth-Order Saddle-Point Algorithm (algname{ZO-SPA}~\cite{sadiev2020}) and Gaussian approximation: deterministic noise setup}
\label{fig:determ_noise}
\end{minipage}
\hfill
\end{figure}

In Figure \ref{fig:determ_noise} we see the convergence of the proposed gradient-free algorithm \algname{ZOS-SEG} and other algorithms applied to the same optimization problem~\eqref{eq:experiments_problems} in the deterministic noise setup (see Subsection \ref{subsec:determ_noise}). We can conclude that proposed algorithm \algname{ZOS-SEG} achieves much better convergence rate than other algorithms. As we can see, with $\Delta = 0.001$ algorithm \algname{ZO-SPA}~\cite{sadiev2020} achieves almost the same rate, however with greater values of parameter $\Delta$ our algorithm prooves its excellence.

\newpage

\begin{figure}[htbp]
\begin{minipage}[h]{\linewidth}
\centering
\includegraphics[width=0.7\linewidth, height=9.5cm]{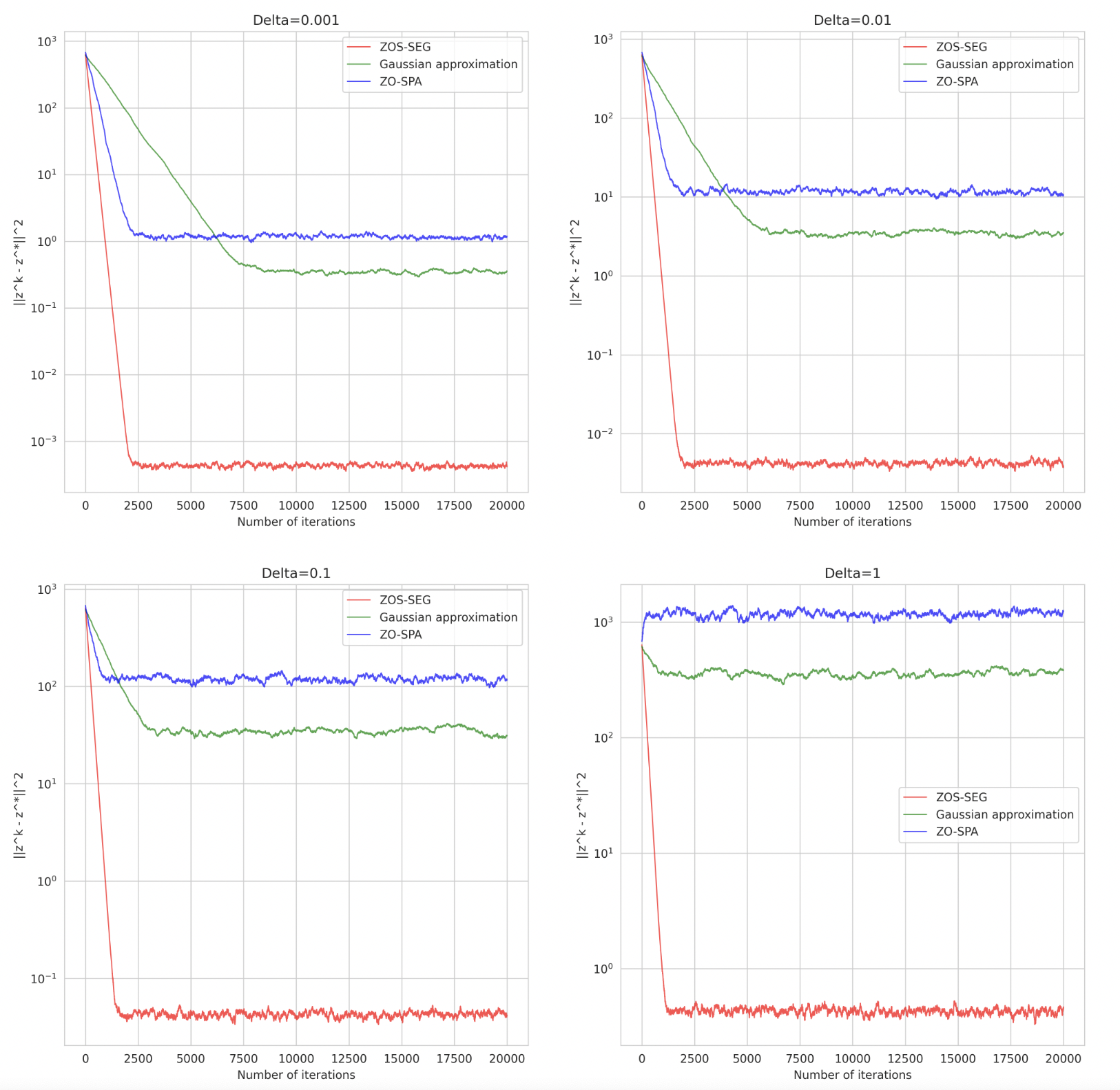}
\captionsetup{justification=centering}
\caption{Convergence of the Zero-Order~Same-sample~Stochastic~Extragradient (\algname{ZOS-SEG}) algorithm, Zeroth-Order Saddle-Point Algorithm (\algname{ZO-SPA}~\cite{sadiev2020}) and Gaussian approximation: stochastic noise setup}
\label{fig:stoch_noise}
\end{minipage}
\hfill

\end{figure}

In Figure \ref{fig:stoch_noise} we see the convergence of the same algorithms in the setup of stochastic noise (see Subsection \ref{subsec:stoch_noise}). We can conclude that proposed Algorithm \algname{ZOS-SEG} achieves the best convergence rate amongst other algorithms, which in this particular setup do not present the needed accuracy. We can also notice, that in both mentioned setups our algorithm has a great convergence rate, however in the stochastic setup, it needs less number of iterations $N$ in order to achieve it.\\
For further details please refer to the code: \url{https://github.com/Sone4ka1567/Gradient_free_algorithm_for_spp}.

\section{Conclusion} \label{sec:Conclusion}

This paper presents \algname{ZOS-SEG} (Algorithm \ref{alg:zosseg}), first  gradient-free algorithm for SPP under overparametrized conditions.
Also we provide the theoretical base for this algorithm: we estimate convergence rate, and number of iterations, required to guarantee that convergence rate achieves $\varepsilon$-accuracy. Moreover, we compared convergence rate estimations in Uniform Sampling setup and Importance Sampling setup: whilst Importance Sampling setup offers improved estimates for number of iterations $N$ and total number of calls $T$, it also introduces challenges in estimating adversarial noise $\Delta$ and batch size $B$ , potentially impacting the reliability of the overall computational algorithm or statistical model. Also we compared deterministic noise setup and stochastic noise setup: estimate of adversarial noise $\Delta$ has increased due to the absence of one of the terms dependent on it. Amongst all the mentioned cases, the Algorithm based on Importance Sampling setup with deterministic noise is the most advantageous based on a combination of factors, such as convergence rate of the algorithm and upper-bound estimation on the adversarial noise $\Delta$. We examine our theoretical results in the experiments and show the convergence rate of our algorithm in comparison with the others. These findings provide a strong foundation for further advancements and development of efficient optimization algorithms for SPP under overparametrization conditions.

\newpage
\subsection*{Acknowledgements}
This work was supported by a grant for research centers in the field of artificial intelligence, provided by the Analytical Center for the Government of the Russian Federation in accordance with the subsidy agreement (agreement identifier 000000D730324P540002) and the agreement with the Moscow Institute of Physics and Technology dated November 1, 2021 No. 70-2021-00138.

\bibliographystyle{cas-model2-names}

\bibliography{cas-refs}

\appendix
\section{Appendix}

\subsection{Basic Inequalities}

For all $a, b, a_i, \in \R^d, i \le n \in \N$ the following inequalities hold:
\begin{equation}
    2\langle a, b\rangle = \|a\|^2 + \|b\|^2 - \|a - b\|^2. \label{eq:inner_product_representation}
\end{equation}
\begin{equation}
    \|a + b\|^2 \geq \frac{1}{2}\|a\|^2 - \|b\|^2, \label{eq:a+b_lower}
\end{equation}
\begin{equation}
    \left\|\sum\limits_{i=1}^n a_i\right\|^2 \leq n\sum\limits_{i=1}^n\|a_i\|^2.\label{eq:a+b}
\end{equation}
Inequality~\eqref{eq:a+b_lower} is obtained from applying the parallelogram law to sum $\|a\|^2 + \|a + 2b\|^2$.\\
Inequality~\eqref{eq:a+b} is obtained from Cauchy–Schwarz inequality and the fact that root mean square is greater or equal than geometric mean.

\subsection{Auxiliary Results}
In this section we list auxiliary facts that we use several times in our~proofs.
    
\begin{lemmasec}[Simplified version of Lemma 3 from Stich  \cite{stich2019unified} , 2019]\label{lem:stich_lemma}
	Let the non-negative sequence $\{r_k\}_{k\ge 0}$ satisfy the relation
	\begin{equation*}
		r_{k+1} \leq (1 - a\gamma_k)r_k - b\gamma_k + c\gamma_k^2
	\end{equation*}
	for all $k \geq 0$, parameters $a, b, c\ge 0$, and any non-negative sequence $\{\gamma_k\}_{k\ge 0}$ such that $\gamma_k \leq \frac{1}{h}$ for some $h \ge a$, $h> 0$. Then, for any $K \ge 0$ one can choose $\{\gamma_k\}_{k \ge 0}$ as follows:
	\begin{eqnarray*}
		\text{if } K \le \frac{h}{a}, && \gamma_k = \frac{1}{h},\\
		\text{if } K > \frac{h}{a} \text{ and } k < k_0, && \gamma_k = \frac{1}{h},\\
		\text{if } K > \frac{h}{a} \text{ and } k \ge k_0, && \gamma_k = \frac{2}{a(\kappa + k - k_0)},
	\end{eqnarray*}
	where $\kappa = \frac{2h}{a}$ and $k_0 = \left\lceil \frac{K}{2} \right\rceil$. For this choice of $\gamma_k$ the following inequality holds:
	\begin{eqnarray*}
		r_{K} \le \frac{32hr_0}{a}\exp\left(-\frac{a K}{2h}\right) + \frac{36c}{a^2 K}.
	\end{eqnarray*}
\end{lemmasec}

\begin{lemmasec}[Wirtinger-Poincare]\label{lem:Wirtinger_Poincare}
Let $f$ be differentiable, then for all $x \in \mathbb{R}^d$, $\tau e \in S^d(\tau)$:
\begin{equation*}
    \expect{f(x+ \tau e)^2} \leq \frac{\tau^2}{d} \expect{\norms{\nabla f(x + \tau e)}^2}.
\end{equation*}
\end{lemmasec}

\subsection{Stochastic extra-gradient algorithm with bias: missing proofs}
In this section, we provide full proofs and missing details from Section~\ref{sec:biased_gradient} on \algname{S-SEG}.

\begin{lemmasec}\label{lem:main_lemma_on_ass3}
	Let Assumption~\ref{as:unified_assumption_general} hold with $A\leq \frac{1}{2}$ and $\rho > C \ge 0$. Then a family of methods~\eqref{eq:general_seg_method} for all $K\ge 0$ has the following convergence rate:
	\begin{equation}
		\Exp \left[\|z^{K+1} - z^*\|^2\right] \leq  (1+C-\rho)\Exp\left[\|z^K - z^*\|^2\right] + D_1 + D_2 + (E_1 + E_2)\zeta^2 - (H_1 + H_2) \zeta R. \label{eq:main_lemma_ass3_one_step}
	\end{equation}
\end{lemmasec}
\begin{proof}
     Since $z^{k+1} = z^k - \gamma_{2, \xi^k} \mathbf{g^{k+\frac{1}{2}}}$, we have
     \begin{eqnarray*}
         \|z^{k+1} - z^*\|^2 &=& \|z^k - \gamma_{2, \xi^k} \mathbf{g^{k+\frac{1}{2}}} - z^*\|^2\\
         &=& \|z^k - z^*\|^2 - 2\gamma_{2, \xi^k}\langle \mathbf{g^{k+\frac{1}{2}}}, z^k - z^*\rangle + \gamma_{2, \xi^k}^2\|\mathbf{g^{k+\frac{1}{2}}}\|^2.
     \end{eqnarray*}
     Taking the expectation, conditioned on $\xi^k$,  using our Assumption~\ref{as:unified_assumption_general} and the definition of\\$P_k = \Exp_{\xi^k}\left[\gamma_{2, \xi^k}\langle \mathbf{g^{k+\frac{1}{2}}}, z^k - z^* \rangle\right]$, we continue our derivation:
     \begin{eqnarray*}
         \Exp_{\xi^k}\left[\|z^{k+1} - z^*\|^2\right] &=& \|z^k - z^*\|^2 - 2P_{k}  + \Exp_{\xi^k}[\gamma_{2, \xi^k}^2\|\mathbf{g^{k+\frac{1}{2}}}\|^2]\\
         &\overset{\eqref{eq:second_moment_bound_general}}{\le}& \|z^k - z^*\|^2 - 2P_{k} + 2AP_{k} + C\|z^k - z^*\|^2 + D_1 + E_1\zeta^2 - H_1 \zeta R\\
         &\overset{A \le \frac{1}{2}}{\leq}& (1+C) \|z^k - z^*\|^2 - P_k + D_1 + E_1\zeta^2 - H_1 \zeta R  \\
         &\overset{\eqref{eq:P_k_general}}{\le}& (1+C-\rho)\|z^k - z^*\|^2 - JG_k + D_1 + D_2 + (E_1 + E_2)\zeta^2 - (H_1 + H_2) \zeta R  .
     \end{eqnarray*}
     Next, we take the full expectation from the both sides
     \begin{equation*}
         \Exp\left[\|z^{k+1} - z^*\|^2\right] \le (1+C-\rho)\Exp\left[\|z^k - z^*\|^2\right] - J\Exp[G_k] + D_1 + D_2 +(E_1 + E_2)\zeta^2 - (H_1 + H_2) \zeta R. 
     \end{equation*}
     If $\rho > C \ge 0$, then in the above inequality we can get rid of the non-positive term $(-J\Exp[G_k])$
    \begin{equation*}
         \Exp\left[\|z^{k+1} - z^*\|^2\right] \le (1+C-\rho)\Exp\left[\|z^k - z^*\|^2\right] + D_1 + D_2 + (E_1 + E_2)\zeta^2 - (H_1 + H_2) \zeta R
    \end{equation*}
    and get \eqref{eq:main_lemma_ass3_one_step}.
\end{proof}

The following two lemmas will be used in the proof of Theorem~\ref{thm:S-SEG_convergence_rate} to establish that Assumption~\ref{as:unified_assumption_general} is satisfied.

\begin{lemmasec}\label{lem:second_moment_bound_S_SEG}
	Let Assumptions~\ref{as:lipschitzness},~\ref{as:str_monotonicity} and~\ref{as:stepsize_and_mu_conditions} hold. If $\gamma_{1,\xi^k}$ satisfies
	\begin{equation*}
	    \gamma_{1,\xi^k} \leq \frac{1}{4|\mu_{\xi^k}| + 2 L_{\xi^k}},
	\end{equation*}
    then $\mathbf{g^{k+\frac{1}{2}}}= F_{\xi^k}(z^{k+\frac{1}{2}}) + b(z^{k+\frac{1}{2}})$, where $z^{k+\frac{1}{2}} = z^k - \gamma_{1,\xi^k}F(z^k) - \gamma_{1,\xi^k}b(z^k)$ satisfies the following inequality:
    \begin{eqnarray}
    \Exp_{\xi^k}\left[\gamma_{1,\xi^k}^2\|\mathbf{g^{k+\frac{1}{2}}}\|^2\right] \leq 8\widehat{P}_{k} + 20 \Exp_{\xi^k}\left[\gamma_{1,\xi^k}^2\|F_{\xi^k}(z^*)\|^2\right] + 10 \Exp_{\xi^k}\left[ \gamma_{1,\xi^k}^2 \zeta^2 \right]  - 8\Exp_{\xi^k}\left[\gamma_{1,\xi^k} \langle z^k - z^*, b(z^{k+\frac{1}{2}})\rangle \right],\label{eq:second_moment_bound_from_lemma_S_SEG}
    \end{eqnarray}
where $\widehat{P}_{k} = \Exp_{\xi^k}\left[\gamma_{1,\xi^k}\langle \mathbf{g^{k+\frac{1}{2}}}, z^k - z^* \rangle\right]$.
\end{lemmasec}
\begin{proof}
Using the step formula $z^{k+1} = z^k - \gamma_{1,\xi^k} \mathbf{g^{k+\frac{1}{2}}} = z^k - \gamma_{1,\xi^k} F(z^{k+\frac{1}{2}}) -  \gamma_{1,\xi^k} b(z^{k+\frac{1}{2}})$, we get:
	\begin{eqnarray}\label{eq:second_moment_bound_S_SEG_technical_1}
		\|z^{k+1} - z^*\|^2 &=& \|z^k - z^*\|^2 -2\gamma_{1,\xi^k}\langle z^k-z^*, \mathbf{g^{k+\frac{1}{2}}} \rangle + \gamma_{1,\xi^k}^2 \|\mathbf{g^{k+\frac{1}{2}}}\|^2\\
		&=& \|z^k - z^*\|^2 - 2\gamma_{1,\xi^k}\left\langle z^k - \gamma_{1,\xi^k} F_{\xi^k}(z^k) - \gamma_{1,\xi^k} b(z^{k}) - z^*, F_{\xi^k} (z^{k+\frac{1}{2}} ) - F_{\xi^k}(z^*) \right\rangle\notag \\
        &&\quad - 2\gamma_{1,\xi^k}^2 \langle F_{\xi^k}(z^k), \mathbf{g^{k+\frac{1}{2}}} - F_{\xi^k}(z^*) \rangle - 2\gamma_{1,\xi^k} \langle z^k - \gamma_{1,\xi^k}F_{\xi^k}(z^{k}) - z^*, b(z^{k+\frac{1}{2}}) \rangle \notag\\
	&&\quad - 2\gamma_{1,\xi^k}^2 \langle b(z^{k}), F_{\xi^k}(z^{k+\frac{1}{2}}) - F_{\xi^k}(z^*) \rangle - 2\gamma_{1,\xi^k}\langle z^k - z^*, F_{\xi^k}(z^*) \rangle + \gamma_{1,\xi^k}^2 \|\mathbf{g^{k+\frac{1}{2}}}\|^2.\notag
	\end{eqnarray}
 Taking the expectation w.r.t. $\xi^k$ from the above identity, using the following equalities:\\\\
 $\Exp_{\xi^k}[\gamma_{1,\xi^k}\langle z^k - z^*, F_{\xi^k}(z^*) \rangle] = \langle z^k - z^*, \Exp_{\xi^k}[\gamma_{1,\xi^k}F_{\xi^k}(z^*)] \rangle \overset{\eqref{eq:S_SEG_AS_stepsize_1}}{=} 0$ ,\\\\
 $\mathbf{g^{k+\frac{1}{2}}}= F_{\xi^k}\left(z^{k} - \gamma_{1,\xi^k} F_{\xi^k}(z^k) - \gamma_{1,\xi^k} b(z^{k})\right) + b(z^{k+\frac{1}{2}})$ and $(\mu_\xi,z^*)$ - strong monotonicity of $F_{\xi}(z)$,\\\\
 we derive the following upper-bound:

	\begin{eqnarray*}
		\Exp_{\xi^k}\left[\|z^{k+1} - z^*\|^2\right] &\leq& \|z^k - z^*\|^2 - 2 \Exp_{\xi^k}\left[\gamma_{1,\xi^k}\langle z^k - \gamma_{1,\xi^k} F_{\xi^k}(z^k) - \gamma_{1,\xi^k} b(z^{k}) - z^*, F_{\xi^k} (z^{k+\frac{1}{2}} ) - F_{\xi^k}(z^*) \rangle\right]\\
        &&\quad  - 2 \Exp_{\xi^k}\left[\gamma_{1,\xi^k}^2 \langle F_{\xi^k}(z^k), \mathbf{g^{k+\frac{1}{2}}} - F_{\xi^k}(z^*) \rangle\right] - 2  \Exp_{\xi^k}\left[\gamma_{1,\xi^k} \langle z^k - \gamma_{1,\xi^k}F_{\xi^k}(z^{k}) - z^*, b(z^{k+\frac{1}{2}}) \rangle\right] \\
		&&\quad - 2\Exp_{\xi^k}\left[\gamma_{1,{\xi^k}}^2\langle b(z^{k}), F_{\xi^k}(z^{k+\frac{1}{2}}) - F_{\xi^k}(z^*) \rangle\right] + \Exp_{\xi^k}\left[\gamma_{1,{\xi^k}}^2\|\mathbf{g^{k+\frac{1}{2}}}\|^2\right] \\
		&\overset{\ref{as:str_monotonicity}}{\le}& \|z^k - z^*\|^2 - 2\Exp_{\xi^k}\left[\gamma_{1,\xi^k}\mu_{\xi^k}\|z^k - \gamma_{1,\xi^k} F_{\xi^k}(z^k) - \gamma_{1,\xi^k} b(z^{k}) - z^*\|^2\right]\\
        &&\quad  - 2 \Exp_{\xi^k}\left[\gamma_{1,\xi^k}^2 \langle F_{\xi^k}(z^k), \mathbf{g^{k+\frac{1}{2}}} - F_{\xi^k}(z^*) \rangle\right] - 2  \Exp_{\xi^k}\left[\gamma_{1,\xi^k} \langle z^k - \gamma_{1,\xi^k}F_{\xi^k}(z^{k}) - z^*, b(z^{k+\frac{1}{2}}) \rangle\right] \\
		&&\quad - 2\Exp_{\xi^k}\left[\gamma_{1,{\xi^k}}^2\langle b(z^{k}), F_{\xi^k}(z^{k+\frac{1}{2}}) - F_{\xi^k}(z^*) \rangle\right] + \Exp_{\xi^k}\left[\gamma_{1,{\xi^k}}^2\|\mathbf{g^{k+\frac{1}{2}}}\|^2\right] \\
		&=& \|z^k - z^*\|^2 - 2\Exp_{\xi^k}\left[\gamma_{1,\xi^k}\mu_{\xi^k}\|z^k - \gamma_{1,\xi^k} F_{\xi^k}(z^k) - \gamma_{1,\xi^k} b(z^{k}) - z^*\|^2\right]\\
		&&\quad  - 2  \Exp_{\xi^k}\left[\gamma_{1,\xi^k} \langle z^k  - z^*, b(z^{k+\frac{1}{2}}) \rangle\right] \\
		&&\quad - 2\Exp_{\xi^k}\left[\gamma_{1,{\xi^k}}^2\langle F_{\xi^k}(z^k) + b(z^{k}), F_{\xi^k}(z^{k+\frac{1}{2}}) - F_{\xi^k}(z^*) \rangle\right] + \Exp_{\xi^k}\left[\gamma_{1,{\xi^k}}^2\|\mathbf{g^{k+\frac{1}{2}}}\|^2\right] \\
        &=& \|z^k - z^*\|^2 - 2\Exp_{\xi^k}\left[\gamma_{1,\xi^k}\mu_{\xi^k}\obf_{\{\mu_{\xi^k} \ge 0\}}\|z^k - \gamma_{1,\xi^k} F_{\xi^k}(z^k) - \gamma_{1,\xi^k} b(z^{k}) - z^*\|^2\right]\\
		&&\quad - 2\Exp_{\xi^k}\left[\gamma_{1,\xi^k}\mu_{\xi^k}\obf_{\{\mu_{\xi^k} < 0\}}\|z^k - \gamma_{1,\xi^k} F_{\xi^k}(z^k) - \gamma_{1,\xi^k} b(z^{k}) - z^*\|^2\right]\\
		&&\quad  - 2  \Exp_{\xi^k}\left[\gamma_{1,\xi^k} \langle z^k  - z^*, b(z^{k+\frac{1}{2}}) \rangle\right] \\
		&&\quad - 2\Exp_{\xi^k}\left[\gamma_{1,{\xi^k}}^2\langle F_{\xi^k}(z^k) + b(z^{k}), F_{\xi^k}(z^{k+\frac{1}{2}}) - F_{\xi^k}(z^*) \rangle\right] + \Exp_{\xi^k}\left[\gamma_{1,{\xi^k}}^2\|\mathbf{g^{k+\frac{1}{2}}}\|^2\right] \\
		&\overset{\eqref{eq:a+b_lower}}{\leq}& \|z^k - z^*\|^2 - \Exp_{\xi^k}[\gamma_{1,\xi^k}\mu_{\xi^k}\obf_{\{\mu_{\xi^k} \ge 0\}}]\|z^k - z^*\|^2\\
		&&\quad + 2\Exp_{\xi^k}[\gamma_{1,\xi^k}^3\mu_{\xi^k}\obf_{\{\mu_{\xi^k} \ge 0\}}\|\mathbf{g^k}\|^2]\\
		&&\quad - 2\Exp_{\xi^k}\left[\gamma_{1,\xi^k}\mu_{\xi^k}\obf_{\{\mu_{\xi^k} < 0\}}\|z^k - z^* - \gamma_{1,\xi^k}\mathbf{g^k}\|^2\right]\\
		&&\quad  - 2  \Exp_{\xi^k}\left[\gamma_{1,\xi^k} \langle z^k  - z^*, b(z^{k+\frac{1}{2}}) \rangle\right] \\
		&&\quad - 2\Exp_{\xi^k}\left[\gamma_{1,{\xi^k}}^2\langle F_{\xi^k}(z^k) + b(z^{k}), F_{\xi^k}(z^{k+\frac{1}{2}}) - F_{\xi^k}(z^*) \rangle\right] + \Exp_{\xi^k}\left[\gamma_{1,{\xi^k}}^2\|\mathbf{g^{k+\frac{1}{2}}}\|^2\right] \\
		&\overset{\eqref{eq:a+b}}{\leq}& \|z^k - z^*\|^2 - \Exp_{\xi^k}[\gamma_{1,\xi^k}\mu_{\xi^k}\obf_{\{\mu_{\xi^k} \ge 0\}} + 4\gamma_{1,\xi^k}\mu_{\xi^k}\obf_{\{\mu_{\xi^k} < 0\}}]\|z^k-z^*\|^2\\
		&&\quad + 2\Exp_{\xi^k}\left[\gamma_{1,{\xi^k}}^3(\mu_{\xi^k}\obf_{\{\mu_{\xi^k} \ge 0\}} - 2\mu_{\xi^k}\obf_{\{\mu_{\xi^k} < 0\}}))\|\mathbf{g^k}\|^2\right]\\
	    &&\quad  - 2  \Exp_{\xi^k}\left[\gamma_{1,\xi^k} \langle z^k  - z^*, b(z^{k+\frac{1}{2}}) \rangle\right] \\
		&&\quad - 2\Exp_{\xi^k}\left[\gamma_{1,{\xi^k}}^2\langle F_{\xi^k}(z^k) + b(z^{k}) , F_{\xi^k}(z^{k+\frac{1}{2}}) - F_{\xi^k}(z^*) \rangle\right] + \Exp_{\xi^k}\left[\gamma_{1,{\xi^k}}^2\|\mathbf{g^{k+\frac{1}{2}}}\|^2\right] \\
		&\overset{\eqref{eq:inner_product_representation}}{=}& \|z^k - z^*\|^2 - \Exp_{\xi^k}[\gamma_{1,\xi^k}\mu_{\xi^k}\obf_{\{\mu_{\xi^k} \ge 0\}} + 4\gamma_{1,\xi^k}\mu_{\xi^k}\obf_{\{\mu_{\xi^k} < 0\}}]\|z^k - z^*\|^2\\
		&&\quad - \Exp_{\xi^k}\left[\gamma_{1,{\xi^k}}^2(1 - 2\gamma_{1,\xi^k}(\mu_{\xi^k}\obf_{\{\mu_{\xi^k} \ge 0\}} - 2\mu_{\xi^k}\obf_{\{\mu_{\xi^k} < 0\}}))\|\mathbf{g^k}\|^2\right]\\
		&&\quad - \Exp_{\xi^k}\left[\gamma_{1,{\xi^k}}^2\|F_{\xi^k}(z^{k+\frac{1}{2}}) - F_{\xi^k}(z^*)\|^2\right]\\
		&&\quad + \Exp_{\xi^k}\left[\gamma_{1,{\xi^k}}^2\|\mathbf{g^k} - F_{\xi^k}(z^{k+\frac{1}{2}}) + F_{\xi^k}(z^*)\|^2\right]\\
  	    &&\quad  - 2  \Exp_{\xi^k}\left[\gamma_{1,\xi^k} \langle z^k  - z^*, b(z^{k+\frac{1}{2}}) \rangle\right]  + \Exp_{\xi^k}\left[\gamma_{1,{\xi^k}}^2\|\mathbf{g^{k+\frac{1}{2}}}\|^2\right] \\
		&\overset{\eqref{eq:S_SEG_AS_stepsize_and_mu}}{\le}& \|z^k -
  z^*\|^2 - \Exp_{\xi^k}\left[\gamma_{1,{\xi^k}}^2(1 - 4\gamma_{1,\xi^k}|\mu_{\xi^k}|)\|\mathbf{g^k}\|^2\right]\\
		&&\quad - \Exp_{\xi^k}\left[\gamma_{1,{\xi^k}}^2\|F_{\xi^k}(z^{k+\frac{1}{2}}) - F_{\xi^k}(z^*)\|^2\right]\\
		&&\quad + \Exp_{\xi^k}\left[\gamma_{1,{\xi^k}}^2\|\mathbf{g^k} - F_{\xi^k}(z^{k+\frac{1}{2}}) + F_{\xi^k}(z^*)\|^2\right]\\
  	    &&\quad  - 2  \Exp_{\xi^k}\left[\gamma_{1,\xi^k} \langle z^k  - z^*, b(z^{k+\frac{1}{2}}) \rangle\right]  + \Exp_{\xi^k}\left[\gamma_{1,{\xi^k}}^2\|\mathbf{g^{k+\frac{1}{2}}}\|^2\right],
	\end{eqnarray*}
where in the last inequality we use $\mu_{\xi^k}\obf_{\{\mu_{\xi^k} \ge 0\}} - 2\mu_{\xi^k}\obf_{\{\mu_{\xi^k} < 0\}} \le 2|\mu_{\xi^k}|$.\footnote{When all $\mu_{\xi} \ge 0$, which is often assumed in the analysis of \algname{S-SEG}, numerical constants in our proof can be tightened. Indeed, in the last step, we can get $- \Exp_{\xi^k}\left[\gamma_{1,{\xi^k}}^2(1 - 2\gamma_{1,\xi^k}\mu_{\xi^k})\|F_{\xi^k}(z^k)\|^2\right]$ instead of $- \Exp_{\xi^k}\left[\gamma_{1,{\xi^k}}^2(1 - 4\gamma_{1,\xi^k}|\mu_{\xi^k}|)\|F_{\xi^k}(z^k)\|^2\right]$.}\\\\
To upper bound the last two terms we use basic inequalities \eqref{eq:a+b_lower} and \eqref{eq:a+b}, and apply $L_{\xi^k}$-Lipschitzness of $F_{\xi^k}(x)$:
    \begin{eqnarray*}
        \Exp_{\xi^k}\left[\|z^{k+1} - z^*\|^2\right] &\overset{\eqref{eq:a+b}}{\leq}& \|z^k - z^*\|^2 - \Exp_{\xi^k}\left[\gamma_{1,\xi^k}^2\left(1 - 4 \gamma_{1,{\xi^k}}|\mu_{\xi^k}|\right)\|\mathbf{g^k}\|^2\right]\\&&
        - \Exp_{\xi^k}\left[\gamma_{1,{\xi^k}}^2\|F_{\xi^k}(z^{k+\frac{1}{2}}) - F_{\xi^k}(z^*)\|^2\right] + 2\Exp_{\xi^k}\left[\gamma_{1,\xi^k}^2\zeta^2\right] \\&& + 2\Exp_{\xi^k}\left[\gamma_{1,\xi^k}^2\|F_{\xi^k}(z^k)-F_{\xi^k}(z^{k+\frac{1}{2}})+F_{\xi^k}(z^*)\|^2\right] 
        \\&& - 2\Exp_{\xi^k}\left[\gamma_{1,\xi^k} \langle z^k - z^*, b(z^{k+\frac{1}{2}})\rangle \right] + \Exp_{\xi^k}\left[\gamma_{1,\xi^k}^2\|\mathbf{g^{k+\frac{1}{2}}}\|^2\right]\\ 
        &\overset{\eqref{eq:a+b}, \eqref{eq:a+b_lower}}{\leq}&
        \|z^k - z^*\|^2 - \Exp_{\xi^k}\left[\gamma_{1,\xi^k}^2\left(1 - 4 \gamma_{1,{\xi^k}}|\mu_{\xi^k}|\right)\|\mathbf{g^k}\|^2\right] \\&& - \frac{1}{2}\Exp_{\xi^k}\left[\gamma_{1,\xi^k}^2 \| F_{\xi^k}(z^{k+\frac{1}{2}})\|^2 \right] + 4\Exp_{\xi^k}\left[ \gamma_{1,\xi^k}^2 \|F_{\xi^k}(z^{k}) - F_{\xi^k}(z^{k+\frac{1}{2}})\|^2\right]\\&& + 5\Exp_{\xi^k}\left[ \gamma_{1,\xi^k}^2 \|F_{\xi^k}(z^{*})\|^2\right] - 2\Exp_{\xi^k}\left[\gamma_{1,\xi^k} \langle z^k - z^*, b(z^{k+\frac{1}{2}})\rangle \right] \\&& + \Exp_{\xi^k}\left[\gamma_{1,\xi^k}^2\|\mathbf{g^{k+\frac{1}{2}}}\|^2\right] + 2\Exp_{\xi^k}\left[\gamma_{1,\xi^k}^2\zeta^2\right]\\ 
         &\overset{\ref{as:lipschitzness}}{\leq}&
         \|z^k - z^*\|^2 - \Exp_{\xi^k}\left[\gamma_{1,\xi^k}^2\left(1 - 4 \gamma_{1,{\xi^k}}|\mu_{\xi^k}|\right)\|\mathbf{g^k}\|^2\right] \\&& -\frac{1}{2}\Exp_{\xi^k}\left[\gamma_{1,\xi^k}^2 \| \mathbf{g^{k+\frac{1}{2}}} - b(z^{k+\frac{1}{2}})\|^2 \right] + 4\Exp_{\xi^k}\left[ \gamma_{1,\xi^k}^4 L_{\xi^k}^2 \|\mathbf{g^k}\|^2\right]\\&& + 5\Exp_{\xi^k}\left[ \gamma_{1,\xi^k}^2 \|F_{\xi^k}(z^{*})\|^2\right] - 2\Exp_{\xi^k}\left[\gamma_{1,\xi^k} \langle z^k - z^*, b(z^{k+\frac{1}{2}})\rangle \right] \\&& + \Exp_{\xi^k}\left[\gamma_{1,\xi^k}^2\|\mathbf{g^{k+\frac{1}{2}}}\|^2\right] + 2\Exp_{\xi^k}\left[\gamma_{1,\xi^k}^2\zeta^2\right]\\ 
        &\overset{\eqref{eq:a+b_lower}}{\leq}&
        \|z^k - z^*\|^2 - \Exp_{\xi^k}\left[\gamma_{1,\xi^k}^2\left(1 - 4 \gamma_{1,{\xi^k}}|\mu_{\xi^k}| - 4\gamma_{1,\xi^k}^2 L_{\xi^k}^2\right)\|\mathbf{g^k}\|^2\right]\\&& + 5\Exp_{\xi^k}\left[ \gamma_{1,\xi^k}^2 \|F_{\xi^k}(z^{*})\|^2\right] - 2\Exp_{\xi^k}\left[\gamma_{1,\xi^k} \langle z^k - z^*, b(z^{k+\frac{1}{2}})\rangle \right] \\&& + \frac{3}{4}\Exp_{\xi^k}\left[\gamma_{1,\xi^k}^2\|\mathbf{g^{k+\frac{1}{2}}}\|^2\right] + \frac{5}{2}\Exp_{\xi^k}\left[\gamma_{1,\xi^k}^2\zeta^2\right]\\ 
        &\overset{\eqref{eq:S_SEG_AS_stepsize_upper_bound}}{\leq}&
        \|z^k - z^*\|^2 + 5\Exp_{\xi^k}\left[ \gamma_{1,\xi^k}^2 \|F_{\xi^k}(z^{*})\|^2\right] - 2\Exp_{\xi^k}\left[\gamma_{1,\xi^k} \langle z^k - z^*, b(z^{k+\frac{1}{2}})\rangle \right] \\&& + \frac{3}{4}\Exp_{\xi^k}\left[\gamma_{1,\xi^k}^2\|\mathbf{g^{k+\frac{1}{2}}}\|^2\right] + \frac{5}{2}\Exp_{\xi^k}\left[\gamma_{1,\xi^k}^2\zeta^2\right]\\
    \end{eqnarray*}\\
Finally, we use the above inequality together with \eqref{eq:second_moment_bound_S_SEG_technical_1}:
    \begin{eqnarray*}
        \| z^k - z^*\|^2 -2\widehat{P}_{k} + \Exp_{\xi^k}\left[\gamma_{1,\xi^k}^2\|\mathbf{g^{k+\frac{1}{2}}}\|^2\right] &\leq&
        \|z^k - z^*\|^2 + 5\Exp_{\xi^k}\left[ \gamma_{1,\xi^k}^2 \|F_{\xi^k}(z^{*})\|^2\right]\\&& - 2\Exp_{\xi^k}\left[\gamma_{1,\xi^k} \langle z^k - z^*, b(z^{k+\frac{1}{2}})\rangle \right]
        + \frac{3}{4}\Exp_{\xi^k}\left[\gamma_{1,\xi^k}^2\|\mathbf{g^{k+\frac{1}{2}}}\|^2\right]\\&& + \frac{5}{2}\Exp_{\xi^k}\left[\gamma_{1,\xi^k}^2\zeta^2\right],
    \end{eqnarray*}
where $\widehat{P}_{k} = \Exp_{\xi^k}\left[\gamma_{1,\xi^k}\langle \mathbf{g^{k+\frac{1}{2}}}, z^k - z^* \rangle\right]$. Rearranging the terms, we obtain \eqref{eq:second_moment_bound_from_lemma_S_SEG}.\\

\end{proof}

\begin{lemmasec}\label{lem:P_k_bound_S_SEG}
    Let Assumptions~\ref{as:lipschitzness},~\ref{as:str_monotonicity} and~\ref{as:stepsize_and_mu_conditions} hold. If $\gamma_{1,\xi^k}$ satisfies \eqref{eq:S_SEG_AS_stepsize_upper_bound}, then\\
    $\mathbf{g^{k+\frac{1}{2}}} = F_{\xi^{k}}(z^{k + \frac{1}{2}}) + b(z^{k+\frac{1}{2}}),$ where $z^{k+\frac{1}{2}} = z^k - \gamma_{1,\xi^k} F_{\xi^k}(z^k) - 
    \gamma_{1,\xi^k} b(z^{k}),$ satisfies the following inequality:
    \begin{equation*}
        \widehat{P}_{k} \ge \widehat{\rho}\|z^k - z^*\|^2 + \frac{1}{2}\widehat{G}_k - \frac{5}{2}\Exp_{\xi^k}\left[\gamma_{1,\xi^k}^2\|F_\xi^k(z^*)\|^2\right]
        -\Exp_{\xi^k}\left[\gamma_{1,\xi^k}^2\zeta^2\right] + \Exp_{\xi^k}\left[\gamma_{1,\xi^k}\langle b(z^{k+\frac{1}{2}}), z^k - z^*\rangle\right],
    \end{equation*}
    where:
    \begin{eqnarray*}
        \widehat{P}_{k} &=& \Exp_{\xi^k}\left[\gamma_{1,\xi^k}\langle \mathbf{g^{k+\frac{1}{2}}}, z^k - z^* \rangle\right],\\
        \widehat{\rho} &=& \frac{1}{2}\Exp_{\xi^k}[\gamma_{1,\xi^k}\mu_{\xi^k}(\obf_{\{\mu_{\xi^k} \ge 0\}} + 4\cdot\obf_{\{\mu_{\xi^k} < 0\}})],\\
        \widehat{G}_k &=& \Exp_{\xi^k}\left[\gamma_{1,\xi^k}^2\left(1 - 4|\mu_{\xi^k}|\gamma_{1,\xi^k} - 4 L_{\xi^k}^2\gamma_{1,\xi^k}^2\right)\|\mathbf{g^k}\|^2\right].
    \end{eqnarray*}
\end{lemmasec}

\begin{proof}
    We start with rewriting $\widehat{P}_{k}$:
    \begin{eqnarray}
        -\widehat{P}_{k} &=& -\Exp_{\xi^k}\left[\gamma_{1,\xi^k}\langle \mathbf{g^{k+\frac{1}{2}}}, z^k - z^*\rangle\right] \overset{\ref{eq:S_SEG_AS_stepsize_1}}{=} -\Exp_{\xi^k}\left[\gamma_{1,\xi^k}\langle F_{\xi^k}(z^{k+\frac{1}{2}}) - F_{\xi^k}(z^*), z^k - z^*\rangle\right] - \Exp_{\xi^k}\left[\gamma_{1,\xi^k}\langle b(z^{k+\frac{1}{2}}), z^k - z^*\rangle\right] \notag\\
        &=& -\Exp_{\xi^k}\left[\gamma_{1,\xi^k}\langle F_{\xi^k}(z^{k+\frac{1}{2}}) - F_{\xi^k}(z^*), z^k - \gamma_{1,\xi^k}F_{\xi^k}(z^k) -  \gamma_{1,\xi^k}b(z^{k}) - z^*\rangle\right] \notag \\&& - \Exp_{\xi^k}\left[\gamma_{1,\xi^k}^2\langle F_{\xi^k}(z^{k+\frac{1}{2}}) - F_{\xi^k}(z^*),F_{\xi^k}(z^k) + b(z^{k})\rangle\right] - \Exp_{\xi^k}\left[\gamma_{1,\xi^k}\langle b (z^{k+\frac{1}{2}}), z^k - z^*\rangle\right] \notag \\&\overset{\eqref{eq:inner_product_representation}} {=}& \underbrace{-\Exp_{\xi^k}\left[\gamma_{1,\xi^k}\langle F_{\xi^k}(z^k - \gamma_{1,\xi^k}F_{\xi^k}(z^k) - \gamma_{1,\xi^k}b(z^{k})) - F_{\xi^k}(z^*), z^k - \gamma_{1,\xi^k}F_{\xi^k}(z^k) - \gamma_{1,\xi^k}b(z^{k}) - z^*\rangle\right]}_{T_1} \notag \\&& \underbrace{-\frac{1}{2}\Exp_{\xi^k}\left[\gamma_{1,\xi^k}^2\|F_{\xi^k}(z^{k+\frac{1}{2}}) - F_{\xi^k}(z^*)\|^2\right] + \frac{1}{2}\Exp_{\xi^k}\left[\gamma_{1,\xi^k}^2\|F_{\xi^k}(z^{k+\frac{1}{2}}) - \mathbf{g^k} - F_{\xi^k}(z^*)\|^2\right]}_{T_2} \notag\\&& -\frac{1}{2}\Exp_{\xi^k}\left[\gamma_{1,\xi^k}^2\|\mathbf{g^k}\|^2\right] - \Exp_{\xi^k}\left[\gamma_{1,\xi^k}\langle b (z^{k+\frac{1}{2}}), z^k - z^*\rangle\right]  . \label{eq:P_k_equation_technical_1}
    \end{eqnarray}

    Next, we upper-bound terms $T_1$ and $T_2$. From $(\mu_{\xi^k}, z^*)$-strong monotonicity of $F_{\xi^k}$ we have: \footnote{When all $\mu_{\xi} \ge 0$, which is often assumed in the analysis of \algname{S-SEG}, numerical constants in our proof can be tightened. Indeed, in the last step of the derivation below, we can get $\Exp_{\xi^k}\left[\mu_{\xi^k}\gamma_{1,\xi^k}^3\|F_{\xi^k}(z^k)\|^2\right]$ instead of $2\Exp_{\xi^k}\left[|\mu_{\xi^k}|\gamma_{1,\xi^k}^3\|F_{\xi^k}(z^k)\|^2\right]$.}
    \begin{eqnarray*}
        T_1 &\overset{\ref{as:str_monotonicity}}{\le}& -\Exp_{\xi^k}\left[\mu_{\xi^k}\gamma_{1,\xi^k}\left\|z^k - z^* - \gamma_{1,\xi^k}F_{\xi^k}(z^k) - \gamma_{1,\xi^k}b(z^{k})\right\|^2\right]\\
        &=& -\Exp_{\xi^k}\left[\obf_{\{\mu_{\xi^k} \ge 0\}}\mu_{\xi^k}\gamma_{1,\xi^k}\left\|z^k - z^* - \gamma_{1,\xi^k}F_{\xi^k}(z^k) - \gamma_{1,\xi^k}b(z^{k})\right\|^2\right]\\
        && -\Exp_{\xi^k}\left[\obf_{\{\mu_{\xi^k} < 0\}}\mu_{\xi^k}\gamma_{1,\xi^k}\left\|z^k - z^* - \gamma_{1,\xi^k}F_{\xi^k}(z^k) - \gamma_{1,\xi^k}b(z^{k})\right\|^2\right]\\ &\overset{\eqref{eq:a+b_lower},\eqref{eq:a+b}}{\le}& 
        -\frac{1}{2}\Exp_{\xi^k}\left[\obf_{\{\mu_{\xi^k} \ge 0\}}\mu_{\xi^k}\gamma_{1,\xi^k}\right]\|z^k - z^*\|^2 + \Exp_{\xi^k}\left[\obf_{\{\mu_{\xi^k} \ge 0\}}\mu_{\xi^k}\gamma_{1,\xi^k}^3\|\mathbf{g^k}\|^2\right]\\
        && -2\Exp_{\xi^k}\left[\obf_{\{\mu_{\xi^k} < 0\}}\mu_{\xi^k}\gamma_{1,\xi^k}\right]\|z^k - z^*\|^2 - 2\Exp_{\xi^k}\left[\obf_{\{\mu_{\xi^k} < 0\}}\mu_{\xi^k}\gamma_{1,\xi^k}^3\|\mathbf{g^k}\|^2\right]\\
        &\le& -\frac{1}{2}\Exp_{\xi^k}\left[\left(\obf_{\{\mu_{\xi^k} \ge 0\}} + 4\cdot\obf_{\{\mu_{\xi^k} < 0\}}\right)\mu_{\xi^k}\gamma_{1,\xi^k}\right]\|z^k - z^*\|^2 + 2\Exp_{\xi^k}\left[|\mu_{\xi^k}|\gamma_{1,\xi^k}^3\|\mathbf{g^k}\|^2\right].
    \end{eqnarray*}

    Using simple inequalities \eqref{eq:a+b_lower} and \eqref{eq:a+b} and applying $L_{\xi^k}$-Lipschitzness of $F_{\xi^k}(x)$, we upper-bound $T_2$:
    \begin{eqnarray*}
        T_2 &\overset{\eqref{eq:a+b_lower},\eqref{eq:a+b}}{\leq}& -\frac{1}{4}\Exp_{\xi^k}\left[\gamma_{1,\xi^k}^2\|F_{\xi^k}(z^{k+\frac{1}{2}})\|^2\right] + \frac{1}{2}\Exp_{\xi^k}\left[\gamma_{1,\xi^k}^2\|F_{\xi^k}(z^*)\|^2\right]\\
        && + \Exp_{\xi^k}\left[\gamma_{1,\xi^k}^2\|F_{\xi^k}(z^k) - F_{\xi^k}(z^{k+\frac{1}{2}}) + F_{\xi^k}(z^*)\|^2\right] + \Exp_{\xi^k}\left[\gamma_{1,\xi^k}^2\|b(z^{k})\|^2\right]\\
        &\overset{\eqref{eq:a+b}}{\leq}& \Exp_{\xi^k}\left[\gamma_{1,\xi^k}^2\zeta^2\right] + 2 \Exp_{\xi^k}\left[\gamma_{1,\xi^k}^2\|F_{\xi^k}(z^k) - F_{\xi^k}(z^{k+\frac{1}{2}})\|^2\right] + 
        \frac{5}{2}\Exp_{\xi^k}\left[\gamma_{1,\xi^k}^2\|F_{\xi^k}(z^*)\|^2\right] \\
        &\overset{\ref{as:lipschitzness}}{\le}& \Exp_{\xi^k}\left[\gamma_{1,\xi^k}^2\zeta^2\right] + \frac{5}{2}\Exp_{\xi^k}\left[\gamma_{1,\xi^k}^2\|F_{\xi^k}(z^*)\|^2\right] +
        2 \Exp_{\xi^k}\left[\gamma_{1,\xi^k}^4 L_{\xi^k}^2\|\mathbf{g^k}\|^2\right].
    \end{eqnarray*}

    Putting all together in \eqref{eq:P_k_equation_technical_1}, we derive:
    \begin{eqnarray*}
        -\widehat{P}_{k} &\le& -\frac{1}{2}\Exp_{\xi^k}\left[\left(\obf_{\{\mu_{\xi^k} \ge 0\}} + 4\cdot\obf_{\{\mu_{\xi^k} < 0\}}\right)\mu_{\xi^k}\gamma_{1,\xi^k}\right]\|z^k - z^*\|^2 \\&&+ 2\Exp_{\xi^k}\left[|\mu_{\xi^k}|\gamma_{1,\xi^k}^3\|\mathbf{g^k}\|^2\right] + \Exp_{\xi^k}\left[\gamma_{1,\xi^k}^2\zeta^2\right] + \frac{5}{2}\Exp_{\xi^k}\left[\gamma_{1,\xi^k}^2\|F_{\xi^k}(z^*)\|^2\right]\\&&+ 2 \Exp_{\xi^k}\left[\gamma_{1,\xi^k}^4 L_{\xi^k}^2\|\mathbf{g^k}\|^2\right]  -\frac{1}{2}\Exp_{\xi^k}\left[\gamma_{1,\xi^k}^2\|\mathbf{g^k}\|^2\right] - \Exp_{\xi^k}\left[\gamma_{1,\xi^k}\langle b (z^{k+\frac{1}{2}}), z^k - z^*\rangle\right]\\ &=&
        - \frac{1}{2}\Exp_{\xi^k}\left[\left(\obf_{\{\mu_{\xi^k} \ge 0\}} + 4\cdot\obf_{\{\mu_{\xi^k} < 0\}}\right)\mu_{\xi^k}\gamma_{1,\xi^k} \right]\|z^k - z^*\|^2\\&&+ \Exp_{\xi^k}\left[\gamma_{1,\xi^k}^2\zeta^2\right] + \frac{5}{2}\Exp_{\xi^k}\left[\gamma_{1,\xi^k}^2\|F_{\xi^k}(z^*)\|^2\right] - \Exp_{\xi^k}\left[\gamma_{1,\xi^k}\langle b (z^{k+\frac{1}{2}}), z^k - z^*\rangle\right] \\&& - \frac{1}{2}\Exp_{\xi^k}\left[\gamma_{1,\xi^k}^2 \left(1 - 4\gamma_{1,\xi^k}|\mu_{\xi^k}| -4\gamma_{1,\xi^k}^2L_{\xi^k}^2\right)\|\mathbf{g^k}\|^2\right],
    \end{eqnarray*}
    where the last term is non-negative due to \eqref{eq:S_SEG_AS_stepsize_upper_bound}. This finishes the proof.
\end{proof}

Combining two previous lemmas with Lemma~\ref{lem:main_lemma_on_ass3}, we derive convergence guarantees for S-SEG with biased oracle.

\begin{theoremsec}[Proof of Theorem~\ref{thm:S-SEG_convergence_rate}]\label{proof:S-SEG_convergence_rate_proof}Let Assumptions~\ref{as:lipschitzness}~and~\ref{as:str_monotonicity} hold. If stepsizes from \ref{eq:S_SEG} update rules satisfy $\gamma_{2,\xi^k} = \alpha \gamma_{1,\xi^k}$, where $\alpha > 0$, and $\gamma_{1,\xi^k}$ satisfies Assumption~\ref{as:stepsize_and_mu_conditions} and \eqref{eq:S_SEG_AS_stepsize_upper_bound}, then $\mathbf{g^{k+\frac{1}{2}}}$ from~\ref{eq:S_SEG} satisfies Assumption~\ref{as:unified_assumption_general} with the following parameters:
	\begin{gather*}
		A = 4\alpha,\quad C = 0,\quad D_1 = 20\alpha^2\Exp_\xi\left[\gamma_{1,\xi}^2\|F_\xi(z^*)\|^2\right] = 20\alpha^2\sigma_{\algname{AS}}^2, \\ E_1 = 10\alpha^2\Exp_\xi\left[\gamma_{1,\xi}^2 \right] = 10\alpha^2 \widehat{\gamma_2} \quad H_1 = 8\alpha^2\Exp_\xi\left[\gamma_{1,\xi} \right] = 8\alpha^2\widehat{\gamma_1},\\
		\rho = \frac{\alpha}{2}\Exp_{\xi^k}[\gamma_{1,\xi^k}\mu_{\xi^k}(\obf_{\{\mu_{\xi^k} \ge 0\}} + 4\cdot\obf_{\{\mu_{\xi^k} < 0\}})],\\
		G_k = \alpha \Exp_{\xi^k}\left[\gamma_{1,\xi^k}^2\left(1 - 4|\mu_{\xi^k}|\gamma_{1,\xi^k} - 4L_{\xi^k}^2\gamma_{1,\xi^k}^2\right)\|\mathbf{g^k}\|^2\right],\quad J = \frac{1}{2},\quad D_2 = \frac{5}{2}\alpha\sigma_{\algname{AS}}^2,\\
        E_2 = \frac{\alpha}{2}\Exp_\xi\left[\gamma_{1,\xi}^2 \right] = \frac{\alpha}{2} \widehat{\gamma_2} \quad H_2 = \alpha \widehat{\gamma_1}
	\end{gather*}
If additionally $\alpha \leq \frac{1}{8}$ and  $ \rho > 0$ ($\rho$ is defined in Assumption \ref{as:unified_assumption_general}), then for all $N > 0$\\(number of steps for \ref{eq:S_SEG} algorithm with the output $z^{N}$): 
	\begin{equation*}
		\Exp\left[\|z^{N+1} - z^*\|^2\right] \leq \left(1 - \rho\right)\Exp\left[\|z^N - z^*\|^2\right] + \frac{5\alpha(8\alpha + 1)}{2}\sigma_{\algname{AS}}^2 + \frac{\alpha(20\alpha + 1)}{2} \widehat{\gamma_2} \zeta^2 - \alpha(8\alpha + 1) \widehat{\gamma_1}\zeta R,
	\end{equation*}
where $z^*$ is the solution of~\eqref{eq:variational-ineq-form}, $\sigma_{\algname{AS}}^2 = \Exp_\xi\left[\gamma_{1, \xi}^2\|F_\xi(z^*)\|^2\right]$, $\widehat{\gamma_1} = \Exp_\xi\left[\gamma_{1,\xi^k}\right] $ and $\widehat{\gamma_2} = \Exp_\xi\left[\gamma_{2,\xi^k}\right]$,\\$\zeta > 0$ is the upper-bound for bias norm (see Definition \ref{def:b_oracle}) and $R =\max_{i\in \{0..N\}}\|z^i - z^*\|.$
\end{theoremsec}
\begin{proof}
    \algname{S-SEG} fits the unified update rule \eqref{eq:general_seg_method} with $\mathbf{g^{k+\frac{1}{2}}} = F_{\xi^k}(z^{k+\frac{1}{2}}) + b(z^{k+\frac{1}{2}})$.\\ Moreover, Lemmas~\ref{lem:second_moment_bound_S_SEG}~and~\ref{lem:P_k_bound_S_SEG} imply that:
    \begin{eqnarray}
		\Exp_{\xi^k}\left[\gamma_{1,\xi^k}^2\| \mathbf{g^{k+\frac{1}{2}}}\|^2\right] &\leq& 8 \widehat{P}_{k} + 20 \Exp_{\xi^k}\left[\gamma_{1,\xi^k}^2\|F_{\xi^k}(z^*)\|^2\right] + 10\zeta^2\widehat{\gamma} - 8\Exp_{\xi^k}\left[\gamma_{1,\xi^k} \langle z^k - z^*, b(z^{k+\frac{1}{2}})\rangle \right]\quad\quad\label{eq:S-SEG_convergence_rate_proof_tech_1}
    \end{eqnarray}
    \vspace{-0.75cm}
    \begin{eqnarray}
        \widehat{P}_{k} &\ge& \widehat{\rho}\|z^k - z^*\|^2 + \frac{1}{2}\widehat{G}_k - \frac{5}{2}\Exp_{\xi^k}\left[\gamma_{1,\xi^k}^2\|F_\xi^k(z^*)\|^2\right] - \frac{1}{2}\zeta^2\widehat{\gamma} + \Exp_{\xi^k}\left[\gamma_{1,\xi^k} \langle z^k - z^*, b(z^{k+\frac{1}{2}})\rangle \right],\label{eq:S-SEG_convergence_rate_proof_tech_2}
    \end{eqnarray}
    where:
    \begin{eqnarray*}
        \widehat{P}_{k} &=& \Exp_{\xi^k}\left[\gamma_{1,\xi^k}\langle \mathbf{g^{k+\frac{1}{2}}}, z^k - z^* \rangle\right],\\
        \widehat{\rho} &=& \frac{1}{2}\Exp_{\xi^k}\left[\gamma_{1,\xi^k}\mu_{\xi^k}(\obf_{\{\mu_{\xi^k} \ge 0\}} + 4\cdot\obf_{\{\mu_{\xi^k} < 0\}}) \right],\\
        \widehat{G}_k &=& \Exp_{\xi^k}\left[\gamma_{1,\xi^k}^2\left(1 - 4|\mu_{\xi^k}|\gamma_{1,\xi^k} - 4 L_{\xi^k}^2\gamma_{1,\xi^k}^2\right)\|\mathbf{g^k}\|^2\right].
    \end{eqnarray*}
    Since $\gamma_{2,\xi^k} = \alpha \gamma_{1,\xi^k}$, we multiply \eqref{eq:S-SEG_convergence_rate_proof_tech_1} by $\alpha^2$ and \eqref{eq:S-SEG_convergence_rate_proof_tech_2} by $\alpha$ and acquire that Assumption~\ref{as:unified_assumption_general} holds with the parameters given in the statement of the theorem. Applying Lemma~\ref{lem:main_lemma_on_ass3} we get the result.
\end{proof}

The next corollary establishes the convergence rate with diminishing stepsizes, allowing to reduce the size of the neighborhood when $\rho > 0$.

\begin{corollarysec}\label{cor:S-SEG-corollary-generic}
    Let Assumptions~\ref{as:lipschitzness},~\ref{as:str_monotonicity}~and~\ref{as:stepsize_and_mu_conditions} hold, $\gamma_{2,\xi^k} = \alpha \gamma_{1,\xi^k}$, where $\alpha = \frac{1}{8}$, $\gamma_{1,\xi^k} = \beta_k \cdot\gamma_{\xi^k}$, and $\gamma_{\xi^k}$ satisfies \eqref{eq:S_SEG_AS_stepsize_upper_bound}. Assume that
	\begin{equation*}
	    \widetilde{\rho} = \frac{1}{16}\Exp_{\xi^k}[ \gamma_{\xi^k}\mu_{\xi^k}(\obf_{\{\mu_{\xi^k} \ge 0\}} + 4\cdot\obf_{\{\mu_{\xi^k} < 0\}})] > 0.
	\end{equation*}
	Then, for all $N \ge 0$ and $\{\beta_k\}_{k\ge 0}$ such that:
	\begin{eqnarray*}
		\text{if } N \le \frac{1}{\widetilde{\rho}}, && \beta_k = 1,\\
		\text{if } N > \frac{1}{\widetilde{\rho}} \text{ and } k < k_0, && \beta_k = 1,\\
		\text{if } N > \frac{1}{\widetilde{\rho}} \text{ and } k \ge k_0, && \beta_k = \frac{2}{2 + \widetilde{\rho}(k - k_0)},
	\end{eqnarray*}
	for $k_0 = \left\lceil\frac{N}{2} \right\rceil$ we have:
	\begin{equation*}
	    \Exp\left[\|z^k - z^*\|^2\right] \le \frac{32 \|x^0 - z^*\|^2}{\widetilde{\rho}} \exp\left( -\frac{\widetilde{\rho}N}{2} \right) + \frac{45}{2N \widetilde{\rho}^2}\sigma_{\algname{AS}}^2 + \frac{63  \widehat{\gamma_2}}{8\widetilde{\rho}^2}\zeta^2 ,
	\end{equation*}
	where $\sigma_{\algname{AS}}^2 = \Exp_\xi\left[\gamma_{\xi}^2\|F_\xi(z^*)\|^2\right]$
\end{corollarysec}
\begin{proof}
    In Theorem~\ref{thm:S-SEG_convergence_rate}, we establish the following recurrence:
    \begin{eqnarray*}
        \Exp\left[\|z^{k+1} - z^*\|^2\right] &\leq& \left(1  - \beta_k\widetilde{\rho}\right)\Exp\left[\|z^k - z^*\|^2\right] + \frac{5\alpha}{2}\left(8\alpha + 1\right)\beta_k^2\sigma_{\algname{AS}}^2 + \frac{\alpha(20\alpha + 1)}{2}\beta_k^2\zeta^2\widehat{\gamma_2} - \alpha(8\alpha + 1)\beta_k \widehat{\gamma_1}\zeta R \\
        &\overset{\alpha = \frac{1}{8}}{=}& \left(1 - \beta_k\widetilde{\rho}\right)\Exp\left[\|z^k - z^*\|^2\right] + \beta_k^2\frac{5\sigma_{\algname{AS}}^2}{8} + \frac{7}{32}\beta_k^2 \zeta^2 \widehat{\gamma_2} - \frac{1}{4}\beta_k \widehat{\gamma_1}\zeta R,
    \end{eqnarray*}
    where we redefined $\rho$ and $\sigma_{\algname{AS}}^2$ to better handle decreasing stepsizes. Applying Lemma~\ref{lem:stich_lemma} for \\$r_k = \Exp\left[\|z^k - z^*\|^2\right]$, $\gamma_k = \beta_k$, $a = \widetilde{\rho}$ , $b = \frac{1}{4}\widehat{\gamma_1}\zeta R$, $c = \frac{5\sigma_{\algname{AS}}^2}{8} + \frac{7}{32}\zeta^2 \widehat{\gamma_2}$, $h = 1$, we get the result. 
\end{proof}

\begin{corollarysec}[Proof of Corollary~\ref{cor:S-SEGUS-o-convergence}]\label{cor:S-SEGUS-o-convergence-proof}Consider the setup from Example~\ref{ex:uniform_sampling}. Let stepsizes from \ref{eq:S_SEG} update rules satisfy $\gamma_{2,\xi^k} = \alpha \gamma_{1,\xi^k}$, where $\alpha = \frac{1}{8}$, and stepsize $\gamma_{1,\xi^k} = \beta_k\gamma = \frac{\beta_k}{6L_{\max}}$, where $L_{\max} = \max_{i\in \{1..n\}}L_i$ and $0 < \beta_k \leq 1$.\\Then for all $N > 0$ (number of steps for \ref{eq:S_SEG} algorithm with the output $z^{N}$) and $\{\beta_k\}_{k\ge 0}$ such that:
	\begin{eqnarray*}
		\text{if } N \le \frac{96L_{\max}}{\overline{\mu}}, && \beta_k = 1,\\
		\text{if } N > \frac{96L_{\max}}{\overline{\mu}} \text{ and } k < k_0, && \beta_k = 1,\\
		\text{if } N > \frac{96L_{\max}}{\overline{\mu}} \text{ and } k \ge k_0, && \beta_k = \frac{96L_{\max}}{96L_{\max} + \overline{\mu}(k - k_0)},
	\end{eqnarray*}
	for $k_0 = \left\lceil\frac{N}{2} \right\rceil$ we have
	\begin{equation*}
	    \Exp\left[\|z^k - z^*\|^2\right] \le \frac{3072\|x^0-z^*\|^2L_{\max}}{\overline{\mu}} \exp\left(-\frac{\overline{\mu} N}{192 L_{\max}}\right) + \frac{5760 \sigma_{\algname{US}*}^2}{\overline{\mu}^2N} + \frac{2016 \zeta^2}{\overline{\mu}^2},
	\end{equation*} which implies that:
        \begin{equation*}
	    \Exp\left[\|z^k - z^*\|^2\right] = \mathcal{O}\left( \frac{L_{\max}R_0^2}{\overline{\mu}}\exp \left( -\frac{\overline{\mu}N}{L_{\max}} \right) + \frac{\sigma_{\algname{US}*}^2}{\overline{\mu}^2 N} + \frac{\zeta^2}{\overline{\mu}^2}\right),
	\end{equation*}
 where $z^*$ is the solution of~\eqref{eq:variational-ineq-form}, $R_0^2 = \|z^0 - z^*\|^2, \sigma_{\algname{US}*}^2 = \frac{1}{n}\sum_{i=1}^n \|F_i(z^*)\|^2$ , $\overline{\mu}$ is defined in Example~\ref{ex:uniform_sampling}\\and $\zeta > 0$ is the upper-bound for bias norm (see Definition \ref{def:b_oracle}).
\end{corollarysec}
\begin{proof}
    Corollary~\ref{cor:S-SEG-corollary-generic} implies the needed result with:
    \begin{eqnarray*}
        \widetilde{\rho} &=& \frac{1}{16}\Exp_{\xi^k}[\gamma_{\xi^k}\mu_{\xi^k}(\obf_{\{\mu_{\xi^k} \ge 0\}} + 4\cdot\obf_{\{\mu_{\xi^k} < 0\}})] = \frac{\gamma}{16n}\left(\sum\limits_{i: \mu_i \geq 0} \mu_i + 4\sum\limits_{i: \mu_i < 0} \mu_i\right) = \frac{\overline{\mu}}{96L_{\max}},\\
        \sigma_{\algname{AS}}^2 &=& \Exp_\xi\left[\gamma_{\xi}^2\|F_\xi(z^*)\|^2\right] = \frac{\gamma^2}{n}\sum_{i=1}^n \|F_i(z^*)\|^2 = \gamma^2 \sigma_{\algname{US}*}^2 = \frac{\sigma_{\algname{US}*}^2}{36 L_{\max}^2},\\
        \widehat{\gamma_1} &=& \Exp_\xi\left[\gamma_{\xi}\right] = \frac{1}{6L_{\max}},\quad
        \widehat{\gamma_2} = \Exp_\xi\left[\gamma_{\xi}^2\right] = \frac{1}{36L_{\max}^2}.
    \end{eqnarray*}
\end{proof}

\begin{corollarysec}[Proof of Corollary~\ref{cor:S-SEGIS-o-convergence}]\label{cor:S-SEGIS-o-convergence-proof}
    Consider the setup from Example~\ref{ex:importance_sampling}. Let $\overline{\mu} > 0$, $\gamma_{2,\xi^k} = \alpha \gamma_{1,\xi^k}$, $\alpha = \frac{1}{8}$, and $\gamma_{1,\xi^k} = \frac{\beta_k\gamma \overline{L}}{L_{\xi^k}} = \frac{\beta_k}{6L_{\xi^k}}$, where $\overline{L} = \frac{1}{n}\sum_{i=1}^n L_i$ and $0 < \beta_k \leq 1$.\\
	Then, for all $N \ge 0$ (number of steps for \ref{eq:S_SEG} algorithm with the output $z^{N}$) and $\{\beta_k\}_{k\ge 0}$ such that:
	\begin{eqnarray*}
		\text{if } N \le \frac{96\overline{L}}{\overline{\mu}}, && \beta_k = 1,\\
		\text{if } N > \frac{96\overline{L}}{\overline{\mu}} \text{ and } k < k_0, && \beta_k = 1,\\
		\text{if } N > \frac{96\overline{L}}{\overline{\mu}} \text{ and } k \ge k_0, && \beta_k = \frac{96\overline{L}}{96\overline{L} + \overline{\mu}(k - k_0)},
	\end{eqnarray*}
	for $k_0 = \left\lceil\frac{N}{2} \right\rceil$ we have
	\begin{equation*}
	    \Exp\left[\|z^K - z^*\|^2\right] \le \frac{3072 \overline{L}\|z^0 - z^*\|^2}{\overline{\mu}}\exp\left(-\frac{\overline{\mu} N}{192 \overline{L}}\right) + \frac{5760\sigma_{\algname{IS*}}^2}{\overline{\mu}^2 N} + \frac{2016 \zeta^2}{\overline{\mu}^2}.
	\end{equation*} which implies that:
        \begin{equation*}
	    \Exp\left[\|z^k - z^*\|^2\right] = \mathcal{O}\left( \frac{\overline{L}R_0^2}{\overline{\mu}}\exp \left( -\frac{\overline{\mu}N}{\overline{L}} \right) + \frac{\sigma_{\algname{IS}*}^2}{\overline{\mu}^2 N} + \frac{\zeta^2}{\overline{\mu}^2}\right),
	\end{equation*}
 where $z^*$ is the solution of~\eqref{eq:variational-ineq-form}, $R_0^2 = \|z^0 - z^*\|^2, \sigma_{\algname{IS}*}^2 = \frac{1}{n}\sum_{i=1}^n \frac{\overline{L}}{L_i}\|F_i(z^*)\|^2$, $\overline{\mu}$ is defined in Example~\ref{ex:importance_sampling}\\and $\zeta > 0$ is the upper-bound for bias norm (see Definition \ref{def:b_oracle}).
\end{corollarysec}
\begin{proof}
    Corollary~\ref{cor:S-SEG-corollary-generic} implies the needed result with:
    \begin{eqnarray*}
        \widetilde{\rho} &=& \frac{1}{16}\Exp_{\xi^k}[\gamma_{\xi^k}\mu_{\xi^k}(\obf_{\{\mu_{\xi^k} \ge 0\}} + 4\cdot\obf_{\{\mu_{\xi^k} < 0\}})] = \frac{\gamma}{16n}\left(\sum\limits_{i: \mu_i \geq 0} \mu_i + 4\sum\limits_{i: \mu_i < 0} \mu_i\right) = \frac{\overline{\mu}}{96\overline{L}},\\
        \sigma_{\algname{AS}}^2 &=& \Exp_\xi\left[\gamma_{\xi}^2\|F_\xi(x^*)\|^2\right] = \frac{\gamma^2}{n}\sum_{i=1}^n \frac{\overline{L}}{L_i}\|F_i(x^*)\|^2 = \gamma^2 \sigma_{\algname{IS}*}^2 = \frac{\sigma_{\algname{IS}*}^2}{36 \overline{L}^2},\\
        \widehat{\gamma_1} &=& \Exp_\xi\left[\gamma_{\xi}\right] = \frac{1}{6\overline{L}},\quad
        \widehat{\gamma_2} = \Exp_\xi\left[\gamma_{\xi}^2\right] = \frac{1}{36\overline{L}^2}.
    \end{eqnarray*}
\end{proof}

\subsection{Main result: missing proofs and assumptions}
In this section we substitute gradient oracle to the gradient approximation. To derive the main convergence result, we need to estimate bias and second moment of gradient approximation:

\begin{lemmasec}[Bias of gradient approximation]\label{lem:bias-of-gradapprox}
Let Assumption~\ref{as:lipschitzness} hold. Then gradient approximation (see~Definition~\ref{def:gradientfreeapprox}) satisfies the following inequality:
    \begin{eqnarray*}
        \|\Exp\left[ g(z_k,\xi,e) \right] - \nabla f(z_k) \| &\leq& L_{\xi}\tau + \frac{d \Delta}{\tau},
    \end{eqnarray*}
    where $d = d_x + d_y$,  $\tau$ is a smoothing parameter (see Definition~\ref{def:gradientfreeapprox}) and $\Delta > 0$ is the upper-bound for bias norm (see Definition \ref{def:gradientfreeoracle}).
\end{lemmasec}
\begin{proof}
    Using the variational representation of the Euclidean norm and Definition \ref{def:gradientfreeapprox} we can write:
    \begin{eqnarray*}
        \|\Exp\left[ g(z_k,\xi,e) \right] - \nabla f(z_k) \| &=&
        \norms{\expect{\frac{d}{2 \tau}\left( \Tilde{f} \left( z_k + \tau e, \xi) - \Tilde{f}(z_k - \tau e, \xi \right) \right) e} - \nabla f(z_k)} \\
        &\overset{\circledOne}{=}& \norms{\expect{\frac{d}{\tau}\left( f \left( z_k + \tau e, \xi) + \delta(z_k + \tau e\right) \right) e} - \nabla f(z_k)} \\
        & \overset{\circledTwo}{\leq}&  \norms{\expect{\frac{d}{\tau} f(z_k + \tau e, \xi) e} - \nabla f(z_k)} + \frac{d \Delta}{\tau} \\
        &\overset{\circledThree}{=}& \norms{\expect{\nabla f(z_k + \tau u, \xi)} - \nabla f(z_k)} + \frac{d \Delta}{\tau} \\
        &=& \sup_{v \in S^d(1)} \expect{| \nabla_v f(z_k + \tau u) - \nabla_v f(z_k)|} + \frac{d \Delta}{\tau} \\
        &\overset{\ref{as:lipschitzness}}{\leq}& L_{\xi} \tau \expect{\norms{u}} + \frac{d \Delta}{\tau} \leq L_{\xi} \tau + \frac{d \Delta}{\tau},
    \end{eqnarray*}
    where $u \in B^d(1)$, $\circledOne =$ the equality is obtained from the fact, namely, distribution of $e$ is symmetric, $\circledTwo =$ the inequality is obtain from bounded noise $|\delta(x)| \leq \Delta$, $\circledThree =$ the equality is obtained from a version of Stokes’ theorem.
\end{proof}

\begin{lemmasec}[Second moment of gradient approximation]\label{lem:bounding-sec-moment-grad-approx}
Let Assumptions~\ref{as:lipschitzness} and \ref{ass:stoch_noise} hold. Then gradient approximation (see~Definition~\ref{def:gradientfreeapprox}) satisfies the following inequality:
    \begin{eqnarray*}
         \expect{\norms{g(z^*,\xi,e)}^2} &\leq& 4d \sigma^2_* + 4 d L_{\xi}^2 \tau^2 + \frac{d^2 \Delta^2}{\tau^2}
    \end{eqnarray*}
    where $d = d_x + d_y$, $\tau$ is a smoothing parameter (see Definition~\ref{def:gradientfreeapprox}) and $\Delta > 0$ is the upper-bound for bias norm (see Definition \ref{def:gradientfreeoracle}).
\end{lemmasec}
\begin{proof}
    By Definition \ref{def:gradientfreeapprox} and Wirtinger-Poincare (see Lemma \ref{lem:Wirtinger_Poincare}) we have:
    \begin{eqnarray*}
        \expect{\norms{g(z^*,\xi,e)}^2} &=& \frac{d^2}{4 \tau^2} \expect{\norms{\left(\Tilde{f}(z^* + \tau e, \xi) - \Tilde{f}(z^* - \tau e, \xi)\right) e}^2} \\
         &=& \frac{d^2}{4 \tau^2} \expect{\left(f(z^* + \tau e, \xi) - f(z^* - \tau e, \xi) + \delta (z^* + \tau e) - \delta (z^* -\tau e)\right)^2}\\
        &\overset{\eqref{eq:a+b}}{\leq}& \frac{d^2}{2 \tau^2} \left( \expect{\left(f(z^* + \tau e, \xi) - f(z^* - \tau e, \xi)\right)^2} + 2 \Delta^2 \right)\\
        &\overset{\ref{lem:Wirtinger_Poincare}}{\leq}& \frac{d^2}{2 \tau^2} \left( \frac{\tau^2}{d} \expect{\norms{ \nabla f(z^* + \tau e, \xi) + \nabla f(z^* - \tau e, \xi)}^2} + 2 \Delta^2 \right) \\
        &=& \frac{d^2}{2 \tau^2} \left( \frac{\tau^2}{d} \expect{\norms{ \nabla f(z^* + \tau e, \xi) + \nabla f(z^* - \tau e, \xi) \pm 2 \nabla f(z^*, \xi)}^2} + 2 \Delta^2 \right) \\
        &\overset{\ref{as:lipschitzness}, \eqref{eq:a+b}}{\leq}& 4d \norms{\nabla f(z^*, \xi)}^2 + 4 d  L_{\xi}^2 \tau^2 \expect{\norms{e}^2}  + \frac{d^2 \Delta^2}{\tau^2} \\
        &\overset{\circledOne}{\leq}& 4d \sigma^2_* + 4 d  L_{\xi}^2 \tau^2 + \frac{d^2 \Delta^2}{\tau^2}, 
    \end{eqnarray*}
    where $\circledOne =$ the inequality obtained from Assumption \ref{ass:stoch_noise}.
\end{proof}

We can now explicitly write down the convergence of the gradient-free SEG-US method.
   
\begin{lemmasec}[Proof of Lemma~\ref{lem:zosseg-convergence-rate}]\label{proof:zosseg-convergence-rate-lemma}
     Substituting upper bounds on the bias (Lemma \ref{lem:bias-of-gradapprox}) and second moment (Lemma \ref{lem:bounding-sec-moment-grad-approx}) for the gradient approximation (see Definition \ref{def:gradientfreeapprox}) in the convergence of the first-order method (Theorem \ref{thm:S-SEG_convergence_rate}).
\end{lemmasec}

\begin{theoremsec}[Proof of Theorem~\ref{th:zosseg-epsilon-acc}]\label{proof:zosseg-epsilon-acc-proof}
        Let Assumptions \ref{as:lipschitzness} and \ref{as:str_monotonicity} hold and let gradient approximation (see~Definition~\ref{def:gradientfreeapprox}) satisfy Assumption \ref{ass:stoch_noise}, then Zero-Order Same-Sample Stochastic Extragradient algorithm with Uniform~Sampling~(see Algorithm~\ref{alg:zosseg}) achieves $\varepsilon$-accuracy: $\expect{ \| x^N - z^* \| ^ 2} \leq \varepsilon$ after: 
        \begin{eqnarray*}
           N &=& \mathcal{O}\left( \frac{L_{\max}}{\overline{\mu}} \ln \left( \frac{R_0^2 L_{\max}}{\varepsilon \overline{\mu}} \right) \right) ;\quad 
           B = \mathcal{O}\left( \max\left\{ \frac{d \overline{\mu}}{L_{\max} \ln \left( \frac{R_0^2 L_{\max}}{\varepsilon \overline{\mu}} \right)} \max \left\{ \frac{\sigma_{*}^2}{\varepsilon \overline{\mu}^2 } ,  1 \right\}, 1 \right\} \right); \\
           \tau &\leq& \frac{\overline{\mu}}{L_{\max}} \sqrt{\varepsilon}; \quad
           \Delta \leq \frac{\varepsilon \overline{\mu}^2}{L_{\max} \sqrt{d}} \min \left\{ \max \left\{ 1 , \frac{ \sigma_{*}}{\overline{\mu} \sqrt{\varepsilon}}, \sqrt{\frac{L_{\max} \ln \left( \frac{R_0^2 L_{\max}}{\varepsilon \overline{\mu}} \right)}{d \overline{\mu}}} \right\}, \frac{1}{\sqrt{d}} \right\}.
        \end{eqnarray*}
        We can now write O-notation for $T = N \cdot B$:
        \begin{eqnarray*}
             T = N \cdot B = \mathcal{O}\left(\max \left\{ d, \frac{d \sigma_{*}^2}{\varepsilon \overline{\mu}^2}, \frac{L_{\max}}{\overline{\mu}} \ln \left( \frac{R_0^2 L_{\max}}{\varepsilon \overline{\mu}} \right) \right\} \right).
        \end{eqnarray*}
\end{theoremsec}
\begin{proof}
Let's assume that we know  variables $L_{\max}, R_0, \overline{\mu}, \sigma_{*}^2, d$. Let us bound the other variables.

    \begin{eqnarray*}
       \expect{ \| x^N - z^* \| ^ 2} 
        & \lesssim \underbrace{\frac{L_{\max}R_0^2}{\overline{\mu}}\exp \left( -\frac{\overline{\mu}N}{L_{\max}} \right)}_{\circledOne}  + \underbrace{\frac{ d \sigma_{*}^2}{\overline{\mu}^2 N B}}_{\circledTwo} + \underbrace{\frac{d L_{\max}^2 \tau^2}{\overline{\mu}^2 N B}}_{\circledThree}
        \nonumber \\
        & \quad \quad   + \underbrace{\frac{d^2 \Delta^2}{\overline{\mu}^2 N B \tau^2}}_{\circledFour} + \underbrace{\frac{L_{\max}^2 \tau^2}{\overline{\mu}^2 }}_{\circledFive} +  \underbrace{\frac{d^2 \Delta^2}{\overline{\mu}^2 \tau^2}}_{\circledSix} .
        \nonumber 
   \end{eqnarray*}
    \textbf{From term $\circledOne$} we find the number of iterations $N$ required for Algorithm~\ref{alg:zosseg} to achieve $\varepsilon$-accuracy:
    \begin{eqnarray}
        \circledOne: \quad \frac{L_{\max}R_0^2}{\overline{\mu}} \exp \left( -\frac{\overline{\mu}N}{L_{\max}} \right) \leq \varepsilon \quad  &\Rightarrow& \quad N \geq  \frac{L_{\max}}{\overline{\mu}} \ln \left( \frac{R_0^2 L_{\max}}{\varepsilon \overline{\mu}} \right) ; \notag \\
         N &=& \mathcal{O}\left( \frac{L_{\max}}{\overline{\mu}} \ln \left( \frac{R_0^2 L_{\max}}{\varepsilon \overline{\mu}} \right) \right). \label{eq:N-notation}
    \end{eqnarray}

    \textbf{From term $\circledFive$} we find the smoothing parameter $\tau$:
    \begin{eqnarray}
        \circledFive: \quad \frac{L_{\max}^2 \tau^2}{\overline{\mu}^2 } \leq \varepsilon \quad &\Rightarrow& \quad \tau \leq  \frac{\overline{\mu}}{L_{\max}} \sqrt{\varepsilon};
         \label{eq:tau-notation}
    \end{eqnarray}

    \textbf{From terms $\circledTwo$ and $\circledThree$} we find the batch size $B$ required to achieve optimality in iteration complexity~$N$: 
    \begin{eqnarray}
        \circledTwo: \quad \frac{ d \sigma_{*}^2}{\overline{\mu}^2 N B} \leq \varepsilon \quad &\Rightarrow& \quad B \geq \frac{d \sigma_{*}^2}{\varepsilon \overline{\mu}^2 N } 
        \overset{\ref{eq:N-notation}}{=} \mathcal{O}\left( \frac{d \sigma_{*}^2}{\varepsilon \overline{\mu} L_{\max} \ln \left( \frac{R_0^2 L_{\max}}{\varepsilon \overline{\mu}} \right)} \right);
        \nonumber\\
        \circledThree: \quad \frac{d L_{\max}^2 \tau^2}{\overline{\mu}^2 N B} \leq \varepsilon \quad &\Rightarrow& \quad B \geq \frac{d L_{\max}^2 \tau^2}{\varepsilon \overline{\mu}^2 N } 
        \overset{\eqref{eq:N-notation}, \eqref{eq:tau-notation}}{=}  \mathcal{O}\left( \frac{d \overline{\mu}}{L_{\max} \ln \left( \frac{R_0^2 L_{\max}}{\varepsilon \overline{\mu}} \right)} \right);
        \nonumber \\
         B &=& \mathcal{O}\left( \max\left\{ \frac{d \overline{\mu}}{L_{\max} \ln \left( \frac{R_0^2 L_{\max}}{\varepsilon \overline{\mu}} \right)} \max \left\{ \frac{\sigma_{*}^2}{\varepsilon \overline{\mu}^2 } ,  1 \right\}, 1 \right\} \right). \label{eq:B-notation}
    \end{eqnarray}

    \textbf{From the remaining terms $\circledFour$ and $\circledSix$}, we find the maximum allowable level of adversarial noise $\Delta$ that still guarantees the convergence to desired accuracy~$\varepsilon$:
    \begin{eqnarray}
        \circledFour: \quad \frac{d^2 \Delta^2}{\overline{\mu}^2 N B \tau^2} \leq \varepsilon \quad &\Rightarrow& \quad \Delta \leq \frac{\varepsilon \overline{\mu}^2 \sqrt{B N}}{d L_{\max}} \overset{\eqref{eq:N-notation}, \eqref{eq:B-notation}}{=} \max \left\{ \frac{\varepsilon \overline{\mu}^2}{L_{\max} \sqrt{d}}, \frac{\overline{\mu} \sigma_{*}}{L_{\max}} \sqrt{\frac{\varepsilon}{d}}, \frac{\varepsilon \overline{\mu}^2}{d}\sqrt{\frac{\ln \left( \frac{R_0^2 L_{\max}}{\varepsilon \overline{\mu}} \right)}{L_{\max} \overline{\mu}}} \right\};
        \nonumber \\
        \circledSix: \frac{d^2 \Delta^2}{\overline{\mu}^2 \tau^2} \leq \varepsilon \quad &\Rightarrow& \quad \Delta \leq \frac{\overline{\mu} \tau \sqrt{\varepsilon}}{d} \quad \overset{\eqref{eq:tau-notation}}{=} \quad \frac{\varepsilon \overline{\mu}^2}{d L_{\max}};
        \nonumber \\
        \Delta &\leq&  \frac{\varepsilon \overline{\mu}^2}{L_{\max} \sqrt{d}} \min \left\{ \max \left\{ 1 , \frac{ \sigma_{*}}{\overline{\mu} \sqrt{\varepsilon}}, \sqrt{\frac{L_{\max} \ln \left( \frac{R_0^2 L_{\max}}{\varepsilon \overline{\mu}} \right)}{d \overline{\mu}}} \right\}, \frac{1}{\sqrt{d}} \right\}. \nonumber
    \end{eqnarray}

    We can now write $\mathcal{O}$-notation for $T = N \cdot B$:
    \begin{eqnarray*}
         T = N \cdot B = \mathcal{O}\left(\max \left\{ d, \frac{d \sigma_{*}^2}{\varepsilon \overline{\mu}^2}, \frac{L_{\max}}{\overline{\mu}} \ln \left( \frac{R_0^2 L_{\max}}{\varepsilon \overline{\mu}} \right) \right\} \right).
    \end{eqnarray*}
\end{proof}

\subsection{Gradient-free oracle with stochastic noise: missing proofs}
In this section we substitute the deterministic noise in the gradient-free oracle to a stochastic noise with bounded second moment. Instead of stochastic value of the objective function $f$ we consider its true value with adding the mentioned stochastic noise. This will lead to a slight change in the convergence rate of Algorithm~\ref{alg:zosseg} and estimation on the adversarial noise $\Delta$. However, the idea behind the proofs remains exactly the same as in the previous section.
\begin{lemmasec}[Bias of gradient approximation with stochastic noise]\label{lem:bias-of-gradapprox-stoch-noise}
Let Assumption~\ref{as:lipschitzness} hold. Then gradient approximation (see~Definition~\ref{def:gradientfreeapprox-stochnoise}) satisfies the following inequality:
    \begin{eqnarray*}
        \|\Exp\left[ g(z_k, \xi, \xi', e) \right] - \nabla f(z_k) \| &\leq& L_{\xi}\tau,
    \end{eqnarray*}
    where $d = d_x + d_y$,  $\tau$ is a smoothing parameter (see Definition~\ref{def:gradientfreeapprox-stochnoise}) and $\Delta > 0$ is the upper-bound for second moment of bias (see Definition \ref{def:gradientfreeoracle-stochnoise}).
\end{lemmasec}
\begin{proof}
    Using the variational representation of the Euclidean norm and Definition \ref{def:gradientfreeapprox-stochnoise} we can write:
    \begin{eqnarray*}
        \|\Exp\left[ g(z_k, \xi, \xi', e) \right] - \nabla f(z_k) \| &=&
        \norms{\expect{\frac{d}{2 \tau}\left( \Tilde{f} \left( z_k + \tau e, \xi) - \Tilde{f}(z_k - \tau e, \xi' \right) \right) e} - \nabla f(z_k)} \\
        &\overset{\circledOne}{=}& \norms{\expect{\frac{d}{\tau} f (z_k + \tau e) e} - \nabla f(z_k) + \expect{\frac{d}{2\tau}(\xi - \xi')e}} \\
        & \overset{\circledTwo}{\leq}&  \norms{\expect{\frac{d}{\tau} f(z_k + \tau e) e} - \nabla f(z_k)} + \frac{d}{2\tau}\norms{\expect{(\xi - \xi')e}} \\
        &\overset{\circledThree}{=}& \norms{\expect{\nabla f(z_k + \tau u)} - \nabla f(z_k)} + \frac{d}{2\tau}\norms{\expect{\xi e} - \expect{\xi' e}}\\
        &\overset{\circledFour}{=}&\sup_{v \in S^d(1)} \expect{| \nabla_v f(z_k + \tau u) - \nabla_v f(z_k)|} + \frac{d}{2\tau}\norms{\Exp \xi \Exp e - \Exp \xi' \Exp e} \\
        &\overset{\ref{as:lipschitzness}}{\leq}& L_{\xi} \tau \expect{\norms{u}} \leq L_{\xi} \tau,
    \end{eqnarray*}
    where $u \in B^d(1)$, $\circledOne =$ the equality is obtained from the fact, namely, distribution of $e$ is symmetric,\\ $\circledTwo =$ the common triangle inequality, $\circledThree =$ the equality is obtained from a version of Stokes’ theorem and the fact that $\xi$ and $\xi'$ are i.i.d., $\circledFour =$ the equality obtained from the fact that $\xi$ and $e$ are independent. 
\end{proof}

\begin{lemmasec}[Second moment of gradient approximation with stochastic noise]\label{lem:bounding-sec-moment-grad-approx-stoch-noise}
Let Assumptions~\ref{as:lipschitzness} and \ref{ass:stoch_noise} hold. Then gradient approximation (see~Definition~\ref{def:gradientfreeapprox-stochnoise}) satisfies the following inequality:
    \begin{eqnarray*}
         \expect{\norms{g(z^*, \xi, \xi', e)}^2} &\leq& 4d \sigma^2_* + 4 d L_{\xi}^2 \tau^2 + \frac{2d^2 \Delta^2}{\tau^2}
    \end{eqnarray*}
    where $d = d_x + d_y$, $\tau$ is a smoothing parameter (see Definition~\ref{def:gradientfreeapprox-stochnoise}) and $\Delta > 0$ is the upper-bound for second moment of bias (see Definition \ref{def:gradientfreeoracle-stochnoise}).
\end{lemmasec}
\begin{proof}
    By Definition \ref{def:gradientfreeapprox-stochnoise} and Wirtinger-Poincare (see Lemma \ref{lem:Wirtinger_Poincare}) we have:
    \begin{eqnarray*}
        \expect{\norms{g(z^*, \xi, \xi', e)}^2} &=& \frac{d^2}{4 \tau^2} \expect{\norms{\left(\Tilde{f}(z^* + \tau e, \xi) - \Tilde{f}(z^* - \tau e, \xi')\right) e}^2} \\
         &=& \frac{d^2}{4 \tau^2} \expect{\left(f(z^* + \tau e) - f(z^* - \tau e) + \xi - \xi'\right)^2}\\
        &\overset{\eqref{eq:a+b}}{\leq}& \frac{d^2}{2 \tau^2} \left( \expect{\left(f(z^* + \tau e) - f(z^* - \tau e)\right)^2} + \expect{(\xi - \xi')^2} \right)\\
        &\overset{\ref{lem:Wirtinger_Poincare}, \eqref{eq:a+b}}{\leq}& \frac{d^2}{2 \tau^2} \left( \frac{\tau^2}{d} \expect{\norms{ \nabla f(z^* + \tau e) + \nabla f(z^* - \tau e)}^2} + 2 \expect{ \xi^2} + 2 \expect{ \xi'^2} \right) \\
        &\leq& \frac{d^2}{2 \tau^2} \left( \frac{\tau^2}{d} \expect{\norms{ \nabla f(z^* + \tau e) + \nabla f(z^* - \tau e) \pm 2 \nabla f(z^*)}^2} + 4 \Delta^2 \right) \\
        &\overset{\ref{as:lipschitzness}, \eqref{eq:a+b}}{\leq}& 4d \norms{\nabla f(z^*)}^2 + 4 d  L_{\xi}^2 \tau^2 \Exp{\norms{e}^2}  + \frac{2d^2 \Delta^2}{\tau^2} \\
        &\overset{\circledOne}{\leq}& 4d \sigma^2_* + 4 d  L_{\xi}^2 \tau^2 + \frac{2d^2 \Delta^2}{\tau^2}, 
    \end{eqnarray*}
    where $\circledOne =$ the inequality obtained from Assumption \ref{ass:stoch_noise}.
\end{proof}

\begin{lemmasec}[Proof of Lemma~\ref{lem:zosseg-convergence-rate-stoch-noise}]\label{proof:zosseg-convergence-rate-lemma-stoch-noise}
     Substituting upper bounds on the bias (Lemma \ref{lem:bias-of-gradapprox-stoch-noise}) and second moment (Lemma \ref{lem:bounding-sec-moment-grad-approx-stoch-noise}) for the gradient approximation (see Definition \ref{def:gradientfreeapprox-stochnoise}) in the convergence of the first-order method (Theorem \ref{thm:S-SEG_convergence_rate}).
\end{lemmasec}

\begin{theoremsec}[Proof of Theorem~\ref{th:zosseg-epsilon-acc-stochnoise}]\label{proof:zosseg-epsilon-acc-proof-stoch-noise}
        Let Assumptions \ref{as:lipschitzness} and \ref{as:str_monotonicity} hold and let gradient approximation (see~Definition~\ref{def:gradientfreeapprox}) satisfy Assumption \ref{ass:stoch_noise}, then Zero-Order Same-Sample Stochastic Extragradient with Uniform~Sampling~(see Algorithm~\ref{alg:zosseg}) achieves $\varepsilon$-accuracy: $\expect{ \| x^N - z^* \| ^ 2} \leq \varepsilon$ after: 
        \begin{eqnarray*}
           N &=& \mathcal{O}\left( \frac{L_{\max}}{\overline{\mu}} \ln \left( \frac{R_0^2 L_{\max}}{\varepsilon \overline{\mu}} \right) \right) ;\quad 
           B = \mathcal{O}\left( \max\left\{ \frac{d \overline{\mu}}{L_{\max} \ln \left( \frac{R_0^2 L_{\max}}{\varepsilon \overline{\mu}} \right)} \max \left\{ \frac{\sigma_{*}^2}{\varepsilon \overline{\mu}^2 } ,  1 \right\}, 1 \right\} \right); \\
           \tau &\leq& \frac{\overline{\mu}}{L_{\max}} \sqrt{\varepsilon}; \quad
           \Delta \leq \frac{\varepsilon \overline{\mu}^2}{L_{\max} \sqrt{d}} \max \left\{ 1 , \frac{ \sigma_{*}}{\overline{\mu} \sqrt{\varepsilon}}, \sqrt{\frac{L_{\max} \ln \left( \frac{R_0^2 L_{\max}}{\varepsilon \overline{\mu}} \right)}{d \overline{\mu}}} \right\}.
        \end{eqnarray*}
        We can now write O-notation for $T = N \cdot B$:
        \begin{eqnarray*}
             T = N \cdot B = \mathcal{O}\left(\max \left\{ d, \frac{d \sigma_{*}^2}{\varepsilon \overline{\mu}^2}, \frac{L_{\max}}{\overline{\mu}} \ln \left( \frac{R_0^2 L_{\max}}{\varepsilon \overline{\mu}} \right) \right\} \right).
        \end{eqnarray*}
\end{theoremsec}
\begin{proof}
As we can see from Lemma~\ref{lem:zosseg-convergence-rate-stoch-noise} convergence rate for Algorithm~\ref{alg:zosseg} with gradient approximation with stochastic noise (see Definition~\ref{def:gradientfreeapprox-stochnoise}) differs from results of Lemma~\ref{lem:zosseg-convergence-rate} only with the abundance of the last term. Therefore, the proof of this theorem remains the same as the proof of Theorem~\ref{th:zosseg-epsilon-acc}, except for the estimation for parameter $\Delta$ which now has less equations, thus leading to abundance of the minimizing condition. 
\end{proof}

\newpage
\subsection{Experiments: parameters fitting}\label{app:experiments}

In this subsection we present graphs for our algorithm with different values of batch size $B$. For the deterministic noise we use the following function: $\delta(z) = \frac{\Delta}{1 + \|z\|}$. For the stochastic noise we use Standard Normal distribution multiplied by $\Delta$.
We optimize $f(x, y)$ \eqref{eq:experiments_problems} with parameters: $d_x = 64, d_y = 64$ (dimensional of problem), $n = 32$
(number of functions in a finite sum), $\tau = 1$
(smoothing parameter), $\gamma = 0.05$ (step size), $\Delta = 0.001$.

\begin{figure}[htbp]
    \centering
    \begin{minipage}{0.48\linewidth}
        \centering
        \includegraphics[width=\linewidth, height=7cm]{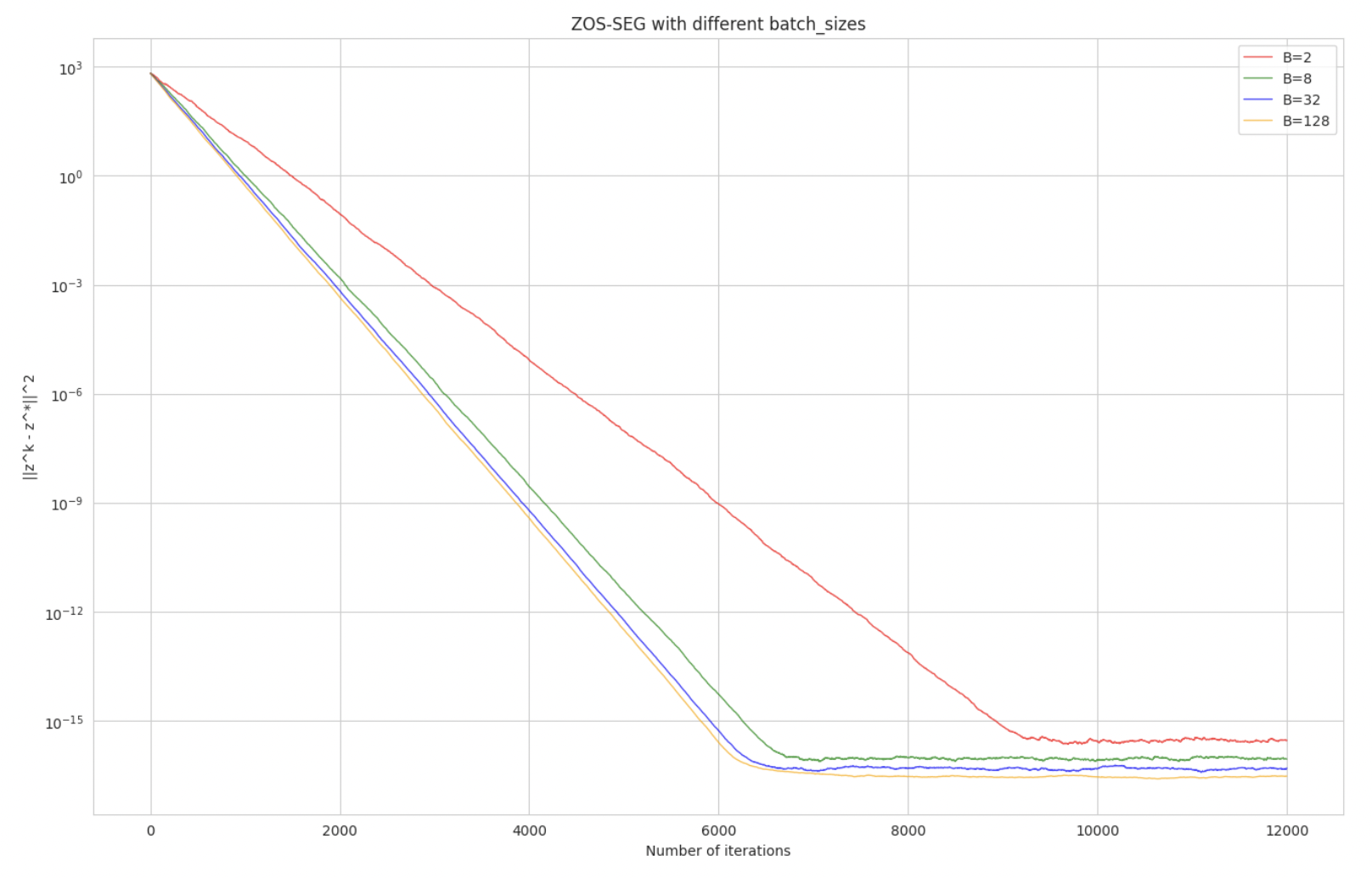}
    \end{minipage}
    \hfill 
    \begin{minipage}{0.48\linewidth}
        \centering
        \includegraphics[width=\linewidth, height=7cm]{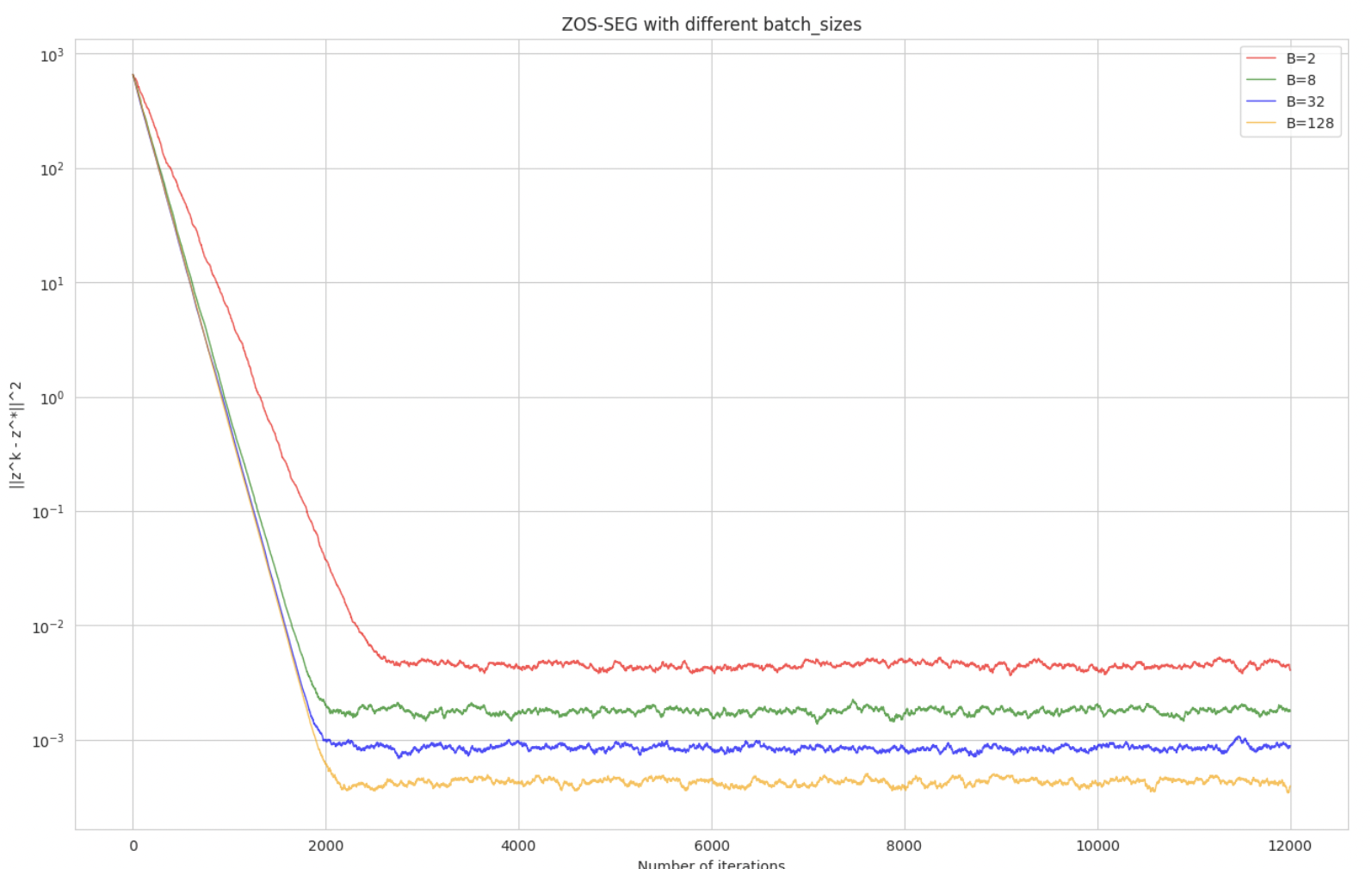}
    \end{minipage}
    \caption{Convergence of the Zero-Order~Same-sample~Stochastic~Extragradient (\algname{ZOS-SEG}) algorithm for different batch sizes $B$: deterministic and stochastic noise setups accordingly.}
    \label{fig:both_noises}
\end{figure}

\end{document}